\documentclass{article}

\usepackage[utf8]{inputenc}
\usepackage[T1]{fontenc}
\usepackage{hyperref}
\usepackage{url}
\usepackage{booktabs}
\usepackage{amsfonts}
\usepackage{nicefrac}
\usepackage{microtype}
\usepackage{lipsum}
\usepackage{graphicx}
\usepackage{amsmath}
\usepackage{amssymb}
\usepackage{amsthm}
\usepackage{enumitem}
\usepackage[all]{xy}
\SelectTips{cm}{12}

\newtheorem{lemma}{Lemma}
\newtheorem*{lemma*}{Lemma}
\newtheorem{prop}{Proposition}
\newtheorem{theorem}{Theorem}
\newtheorem*{theorem*}{Theorem}

\newcommand{\dotminus}{\mathbin{\dot{-}}}

\newcommand{\dotoplus}{\mathbin{\dot{\oplus}}}

\DeclareMathOperator\unit{U}
\DeclareMathOperator\eunit{EU}
\DeclareMathOperator\stunit{StU}
\DeclareMathOperator\kunit{KU}
\DeclareMathOperator\stmap{st}

\DeclareMathOperator\diag{D}

\DeclareMathOperator\glin{GL}

\DeclareMathOperator\symp{Sp}

\DeclareMathOperator\orth{O}

\newcommand{\eps}{\varepsilon}
\newcommand{\leqt}{\trianglelefteq}
\newcommand{\inv}[1]{\!\;\overline{\!\!\:#1\vphantom !\!\!\:}\;\!}
\newcommand{\op}{{\mathrm{op}}}
\DeclareMathOperator{\id}{id}
\DeclareMathOperator{\Aut}{Aut}
\DeclareMathOperator{\Heis}{Heis}

\newcommand{\up}[2]{{^{#1}\!{#2}}}

\makeatletter
\newcommand{\bigperp}{\mathop{\mathpalette\bigp@rp\relax}\displaylimits}
\newcommand{\bigp@rp}[2]{\vcenter{\m@th\hbox{\scalebox{\ifx#1\displaystyle2.1\else1.5\fi}{\(#1\perp\)}}}}
\makeatother

\newcommand{\Set}{\mathbf{Set}}
\newcommand{\Group}{\mathbf{Grp}}
\DeclareMathOperator{\Pro}{Pro}

\title{Centrality of odd unitary \(K_2\)-functor}

\author{
  Egor Voronetsky
  \thanks{Research is supported by the Russian Science Foundation grant 19-71-30002.} \\
  Chebyshev Laboratory, \\
  St. Petersburg State University, \\
  14th Line V.O., 29B, \\
  Saint Petersburg 199178 Russia \\
}

\begin{document}
\maketitle

\begin{abstract}
Let \((R, \Delta)\) be an odd form algebra. We show that the unitary Steinberg group \(\stunit(R, \Delta)\) is a crossed module over the odd unitary group \(\unit(R, \Delta)\) in two major cases: if the odd form algebra has a free orthogonal hyperbolic family satisfying a local stable rank condition and if the odd form algebra is sufficiently isotropic and quasi-finite. The proof uses only elementary localization techniques in terms of pro-groups.
\end{abstract}

\section{Introduction}

Centrality of \(\mathrm K_2\) for the general linear group over arbitrary commutative ring \(K\) was proved in \cite{AnotherPresentation} by W. van der Kallen. His proof actually shows that \(\mathrm{St}(n, K) \to \glin(n, K)\) is a crossed module for sufficiently large \(n\). In \cite{Tulenbaev} M.\,S. Tulenbaev generalized this result for all almost commutative \(K\).

In \cite{CentralityC, CentralityD, CentralityE} S. Sinchuk and A. Lavrenov proved similar results for the Chevalley groups of type \(\mathsf C_l\), \(\mathsf D_l\), and \(\mathsf E_l\). These proofs use the so-called ``another presentation'' of the Steinberg group in terms of non-elementary transvections or a reduction to the linear case. Also, S. B\"oge showed in \cite{AnotherPresentationOrth} that all sufficiently isotropic orthogonal Steinberg groups over fields not of characteristic \(2\) admit ``another presentation''.

For even unitary matrix groups in the sense of A. Bak centrality was announced in 1998 by Bak and G. Tang, but this result is unpublished. In the local case A. Stavrova proved centrality in \cite{Stavrova} for all sufficiently isotropic reductive groups, including the classical groups. Also centrality easily follows from surjective stability of \(\kunit_2\). See \cite{StabVor} and \cite{Weibo} for injective stability of \(\kunit_1\) in the odd unitary case, the proofs from these papers also give surjective stability of \(\kunit_2\).

In \cite{LinK2} we proved that \(\mathrm{St}(R) \to \glin(R)\) is a crossed module for any almost commutative ring \(R\) with a complete family of full orthogonal idempotents, i.e. for isotropic linear groups. Our proof also works for matrix algebras over rings with small local stable rank. In the isotropic case there is no natural notion of unimodular vectors, hence we cannot even formulate van der Kallen's approach in such generality. Together with Lavrenov and Sinchuk we generalized this result in \cite{ChevK2} for all simple simply connected Chevalley groups of all types with the exceptions \(\mathsf B_l\) and \(\mathsf C_l\).

The related result on normality of the elementary subgroup is known in all these cases: for Chevalley groups this was proved in \cite{Taddei} by G. Taddei, for odd unitary groups in \cite{OddDefPetrov} by V. Petrov, and for isotropic reductive groups in \cite{PetrovStavrova} by Petrov and Stavrova.

Odd unitary groups were defined in \cite{OddDefPetrov} by Petrov. These groups generalize even unitary groups and most classical groups. They are defined in terms of odd form parameters on modules with hermitian forms, so they have geometric nature but it is hard to apply algebraic constructions such as Stein's relativization or faithfully flat descent. We discovered odd form algebras in \cite{StabVor}. This objects are more flexible than Petrov's modules with odd form parameters and they still have a natural notion of unitary groups and elementary transvections.

Our goal is to generalize results from \cite{LinK2} and \cite{ChevK2} to odd unitary groups. The main result is
\begin{theorem*}
Let \(K\) be a commutative ring, \((R, \Delta)\) be an odd form \(K\)-algebra with an orthogonal hyperbolic family of rank \(n\). Suppose that \(n \geq 4\) or \(n \geq 3\) and the orthogonal hyperbolic family is strong. Suppose also that \(R\) is quasi-finite over \(K\) or \(\Lambda\mathrm{lsr}(\eta_1, R, \Delta) \leq n - 1\) and \(\mathrm{lsr}(R_{11}) \leq n - 2\). Then there is unique action of \(\unit(R, \Delta)\) on \(\stunit(R, \Delta)\) making \(\stmap\) a crossed module, it is consistent with the action of \(\stunit(R, \Delta) \rtimes \diag(R, \Delta)\).
\end{theorem*}

In particular, \(\mathrm{StO}(2l, K) \to \orth(2l, K)\), \(\mathrm{StSp}(2l, K) \to \symp(2l, K)\), and \(\mathrm{StO}(2l + 1, K) \to \orth(2l + 1, K)\) are crossed modules in a unique way for any commutative ring \(K\) and \(l \geq 3\). For \(l = 2\) there are counterexamples, see \cite{Wendt}.

Together with Lavrenov's result \cite{OddUnitaryLavrenov} that the odd unitary Steinberg group is universally closed our result gives an explicit presentation of the universal central extension of the elementary unitary group.

The paper is organized as follows. In section \(2\) we briefly recall definitions from \cite{StabVor} concerning odd form algebras. Unfortunately, there is not a natural notion for their homotopes, so we give a new definition of augmented odd form algebras in section \(3\) together with constructions of their localizations, homotopes, and odd form pro-algebras. In section \(4\) we consider hyperbolic pairs from \cite{StabVor} and corresponding subgroups of the unitary group: the opposite maximal parabolic subgroups, their common Levi subgroup, and their unipotent radicals. Section \(5\) contains definitions of odd unitary Steinberg groups and pro-groups. In section \(6\) we prove several technical lemmas in order to construct morphisms between our pro-groups. Sections \(7\)--\(9\) are devoted to the proof of ``root elimination'', i.e. that we may remove one root from the definition of the Steinberg group. In section \(10\) we construct actions of local Steinberg groups on the corresponding Steinberg pro-groups, and in the last section we prove the main result.

\section{Odd form algebras}

For group operations we use the conventions \(\up xy = x y x^{-1}\), \(x^y = y^{-1} x y\), and \([x, y] = xyx^{-1} y^{-1}\). A precrossed module is a group homomorphism \(d \colon H \to G\) with an action of \(G\) on \(H\) by automorphisms (denoted by \(\up gh\)) such that \(d(\up gh) = \up g{d(h)}\). A crossed module is a precrossed module with the property \(\up{d(h)}{h'} = \up h{h'}\). For every crossed module the kernel of \(d\) is a central subgroup of \(H\) and the image is a normal subgroup of \(G\).

We usually use the symbol \(\dotplus\) for the group operation in a \(2\)-step nilpotent group. In this case \(x \dotminus y = x \dotplus (\dotminus y)\), where \(\dotminus x\) is the inverse element, and \(\dot 0\) is the identity element. If \(\Delta_i\) are subgroups of such a group satisfying some obvious conditions, then \(\bigoplus^\cdot_i \Delta_i\) is their iterated inner semi-direct product.

All rings in this paper are associative, but not necessarily unital. All commutative rings as unital, as well as homomorphisms between them. An element \(x\) in a non-unital ring \(R\) is called quasi-invertible if it is invertible under the monoidal operation \(x \circ y = xy + x + y\). If \(R\) has an identity, then the group of quasi-invertible elements of \(R\) is isomorphic to the group \(R^*\) of invertible elements by \(x \mapsto x + 1\). By \(R^\bullet\) we denote the multiplicative semigroup of \(R\).

Recall the definitions from \cite{OddDefPetrov, StabVor}. Let \(S\) be a unital ring and \(\lambda \in S^*\). A map \(\inv{(-)} \colon S \to S\) is called a \(\lambda\)-involution if it is an anti-automorphism of \(S\) with \(\inv{\inv s} = \lambda s \lambda^{-1}\) and \(\inv \lambda = \lambda^{-1}\). An involution is the same as a \(1\)-involution, in this case we do not need that \(S\) is unital. Fix a ring \(S\) with a \(\lambda\)-involution. A hermitian form on a module \(M_S\) is a biadditive map \(B \colon M \times M \to S\) such that \(B(m, m's) = B(m, m') s\) and \(B(m', m) = \inv{B(m, m')} \lambda\).

Consider a module \(M_S\) with a hermitian form \(B\). Its Heisenberg group is the set \(\Heis(B) = M \times S\) with the group operation \((m, x) \dotplus (m', x') = (m + m', x - B(m, m') + x')\). This group has a right \(S^\bullet\)-action \((m, x) \cdot s = (ms, \inv s x s)\). An odd form parameter is an \(S^\bullet\)-subgroup \(\mathcal L \leq \Heis(B)\) such that \(\mathcal L_{\mathrm{min}} \leq \mathcal L \leq \mathcal L_{\mathrm{max}}\), where \(\mathcal L_{\mathrm{min}} = \{(0, x - \inv x \lambda) \mid x \in S\}\) and \(\mathcal L_{\mathrm{max}} = \{(m, x) \mid B(m, m) + x + \inv x \lambda = 0\}\). For any odd form parameter the map \(q \colon M \to \Heis(B) / \Lambda, m \mapsto (m, 0) \dotplus \Lambda\) is called a quadratic form. A unitary group of \(M\) consists of the automorphisms of \(M\) stabilizing both hermitian and quadratic forms. An \(S\)-module with fixed hermitian and quadratic forms is called a quadratic \(S\)-module.

It turns out that for any quadratic \(S\)-module there is a ring \(R\) with an involution and an odd form parameter \(\Delta \leq \Heis(B_R)\) (where \(B_R(a, b) = \inv ab\)) such that its unitary group is isomorphic to the unitary group of \(M\), see \cite{StabVor}. Such a pair \((R, \Delta)\) is the same as a special unital odd form ring according to the definition below if we consider \(\pi\) and \(\rho\) are the first and the second projections from \(\Delta\) to \(R\), and take \(\phi(a) = (0, a - \inv a)\).

An odd form ring is a pair \((R, \Delta)\), where \(R\) is a non-unital ring with an involution \(x \mapsto \inv x\), \(\Delta\) is a group (with the group operation \(\dotplus\)), the semi-group \(R^\bullet\) acts on \(\Delta\) from the right by endomorphisms (the action is denoted by \(u \cdot a\)), and there are fixed maps \(\phi \colon R \to \Delta\), \(\pi \colon \Delta \to R\), \(\rho \colon \Delta \to R\) such that for all \(u, v \in \Delta\) and \(a, b \in R\)
\begin{itemize}
\item \(\pi(u \dotplus v) = \pi(u) + \pi(v)\), \(\pi(u \cdot a) = \pi(u) a\);
\item \(\phi(a + b) = \phi(a) \dotplus \phi(b)\), \(\phi(b) \cdot a = \phi(\inv a b a)\);
\item \(\rho(u \dotplus v) = \rho(u) - \inv{\pi(u)} \pi(v) + \rho(v)\), \(\rho(u \cdot a) = \inv a \rho(u) a\);
\item \(\rho(u) + \inv{\rho(u)} + \inv{\pi(u)} \pi(u) = 0\);
\item \(\pi(\phi(a)) = 0\), \(\rho(\phi(a)) = a - \inv a\);
\item \([u, v] = \phi(-\inv{\pi(u)} \pi(v))\);
\item \(\phi(a) = \dot 0\) if \(a = \inv a\);
\item \(u \cdot (a + b) = u \cdot a \dotplus \phi(\inv{\,b\,} \rho(u) a) \dotplus u \cdot b\).
\end{itemize}
It follows that \(\Delta\) is \(2\)-step nilpotent (\(\phi(R) \leqt \Delta\) is a central subgroup), \(\rho(\dot 0) = 0\), \(\rho(\dotminus u) = \inv{\rho(u)}\), and \(u \cdot 0 = \dot 0\). An odd form ring is called unital if \(R\) is unital and \(u \cdot 1 = u\) for all \(u \in \Delta\). An odd form ring is called special if \((\pi, \rho) \colon \Delta \to R \times R\) is injective.

A unitary group of an odd form ring is a certain subset of \(R \times \Delta\). We denote the compotents of its elements \(g\) by \(\beta(g) \in R\) and \(\gamma(g) \in \Delta\), so \(g = (\beta(g), \gamma(g))\). Also let \(\alpha(g) = \beta(g) + 1 \in R \rtimes K\). A unitary group is
\[\unit(R, \Delta) = \{g \in R \times \Delta \mid \alpha(g)^{-1} = \inv{\alpha(g)}, \pi(\gamma(g)) = \beta(g), \rho(\gamma(g)) = \inv{\beta(g)}\}.\]
The group operation is given by \(\alpha(gg') = \alpha(g)\, \alpha(g')\) and \(\gamma(gg') = \gamma(g) \cdot \alpha(g') \dotplus \gamma(g')\). In the special unital case \(\unit(R, \Delta) \cong \{g \in R^* \mid g^{-1} = \inv g, (g - 1, \inv g - 1) \in \Delta\}\), so this definition is consistent with unitary groups of quadratic modules. The unitary group acts on \((R, \Delta)\) by automorphisms in the following way:
\[\up g a = \alpha(g)\, a\, \inv{\alpha(g)}, \quad \up g u = (\gamma(g) \cdot \pi(u) \dotplus u) \cdot \inv{\alpha(g)}.\]

A morphism \(f \colon (R, \Delta) \to (S, \Theta)\) of odd form rings consists of maps \(f \colon R \to S\) and \(f \colon \Delta \to \Theta\) preserving all operations. An odd form ideal in an odd form ring \((R, \Delta)\) is a pair \((I, \Gamma)\) such that \(I = \inv{\,I\,} \leqt R\) is an ideal, \(\Gamma \leqt \Delta\) is an \(R^\bullet\)-subgroup, and \(\Delta \cdot I \dotplus \phi(\{a \in R \mid a - \inv a \in I\}) \leq \Gamma \leq \{u \in \Delta \mid \pi(u), \rho(u) \in I\}\). If \((I, \Gamma) \leqt (R, \Delta)\), then \((R / I, \Delta / \Gamma)\) is also an odd form ring. We say that the sequence \((I, \Gamma) \to (R, \Delta) \to (R / I, \Delta / \Gamma)\) is short exact.

An odd form ring \((R, \Delta)\) acts on an odd form ring \((S, \Theta)\) if there are multiplication maps \(R \times S \to S\), \(S \times R \to S\), \(\Theta \times R \to \Theta\), and \(\Delta \times S \to \Theta\) such that for all \(a, a' \in R\); \(b, b' \in S\); \(u, u' \in \Delta\); and \(v, v' \in \Theta\)
\begin{itemize}
\item \(ab = \inv{\,b\,} \inv a\);
\item \((a + a') b = ab + a'b\), \(a(b + b') = ab + ab'\);
\item \((aa') b = a(a'b)\), \((ab) b' = a(bb')\), \((ab)a' = a(ba')\), \((ba)b' = b(ab')\);
\item \((u \dotplus u') \cdot b = u \cdot b \dotplus u' \cdot b\), \(u \cdot (b + b') = u \cdot b \dotplus \phi(\inv{\,b'\,} \rho(u) b) \dotplus u \cdot b'\);
\item \((v \dotplus v') \cdot a = v \cdot a \dotplus v' \cdot a\), \(v \cdot (a + a') = v \cdot a \dotplus \phi(\inv{a'} \rho(v) a) \dotplus v \cdot a'\);
\item \((u \cdot a) \cdot b = u \cdot ab\), \((u \cdot b) \cdot a = u \cdot ba\), \((u \cdot b) \cdot b' = u \cdot bb'\);
\item \((v \cdot a) \cdot b = v \cdot ab\), \((v \cdot b) \cdot a = v \cdot ba\), \((v \cdot a) \cdot a' = v \cdot aa'\);
\item \(\phi(a) \cdot b = \phi(\inv{\,b\,} a b)\), \(\phi(b) \cdot a = \phi(\inv a b a)\);
\item \(\pi(u \cdot b) = \pi(u) b\), \(\pi(v \cdot a) = \pi(v) a\);
\item \(\rho(u \cdot b) = \inv{\,b\,} \rho(u) b\), \(\rho(v \cdot a) = \inv a \rho(v) a\).
\end{itemize}
An action of a unital odd form ring \((R, \Delta)\) on an odd form ring \((S, \Theta)\) is called unital if \(b1 = b = 1b\) for all \(b \in S\) and \(v \cdot 1 = v\) for all \(v \in \Theta\).

Let \(K\) be a commutative ring. Odd form \(K\)-algebras from \cite{StabVor} are the same as odd form rings \((R, \Delta)\) with a unital action of the odd form ring \((K, \dot 0)\) (with the trivial involution on \(K\)) such that \(ak = ka\) for all \(a \in R\), \(k \in K\). Every odd form ring is an odd form \(\mathbb Z\)-algebra with a uniquely determined multiplication \(\Delta \times \mathbb Z \to \Delta\).

\begin{lemma}\label{ofa-semidirect}
If an odd form ring \((R, \Delta)\) acts on \((S, \Theta)\), then \((S \rtimes R, \Theta \rtimes \Delta)\) is also an odd form ring, where \(S \rtimes R = S \times R\) and \(\Theta \rtimes \Delta = \Theta \times \Delta\) as sets, the operations are
\begin{itemize}
\item \((b, a) + (b', a') = (b + b', a + a')\), \((b, a) (b', a') = (bb' + ab' + ba', aa')\), \(\inv{(b, a)} = (\inv{\,b\,}, \inv a)\);
\item \((v, u) \cdot (b, a) = \bigl(v \cdot b + v \cdot a + u \cdot b + \phi\bigl(\inv a \rho(v) b + \inv a \rho(u) b\bigr), u \cdot a\bigr)\);
\item \(\phi(b, a) = (\phi(b), \phi(a))\), \(\pi(v, u) = (\pi(v), \pi(u))\), \(\rho(v, u) = \bigl(\rho(v) - \inv{\pi(v)} \pi(u), \rho(u) \bigr)\)
\end{itemize}
for \(a, a' \in R\); \(b, b' \in S\); \(u, u' \in \Delta\); \(v, v' \in \Theta\). The sequence \((S, \Theta) \to (S \rtimes R, \Theta \rtimes \Delta) \to (R, \Delta)\) is short exact. If the action is unital, then the semi-direct product of odd form rings is also unital.

A short exact sequence \((I, \Gamma) \to (R, \Delta) \to (S, \Theta)\) splits (i.e. there is a section \((S, \Theta) \to (R, \Delta)\)) if and only if \((R, \Delta) \cong (I \rtimes S, \Gamma \rtimes \Theta)\) as an extension of odd form rings.
\end{lemma}
\begin{proof}
Consider a universal situation, where \((R, \Delta)\) is a free odd form ring and \((S, \Theta)\) is a free odd form ring with an action of \((R, \Delta)\) (it is not necessarily a free abstract odd form ring). The axioms on action ensure that the semi-direct product as a pair of sets is exactly the free odd form ring generated by all generators of \((R, \Delta)\) and \((S, \Theta)\), see proposition \(1\) in \cite{StabVor}. Hence the semi-direct product is always an odd form ring. All other claims follow from direct computations.
\end{proof}

\section{Localization and homotopes}

We use the notation from \cite{ChevK2, LinK2}. Let \(\mathcal C\) be a category. Objects in its pro-completion \(\Pro(\mathcal C)\) are written with an upper index \((\infty)\), such as \(X^{(\infty)}\). Projective limits in \(\Pro(\mathcal C)\) are denoted by \(\varprojlim^{\Pro}\). There is a ``forgetful'' functor from the category of pro-groups \(\Pro(\Group)\) to the category of group objects in the category of pro-sets \(\Pro(\Set)\). This functor is fully faithful, so we identify pro-groups with the corresponding pro-sets. But this functor is not an equivalence: it is easy to see that the ``local group'' \(I^{(\infty)} = \varprojlim_{n \in \mathbb N}^{\Pro} (-\frac 1n, \frac 1n)\), where \((-\frac 1n, \frac 1n) \subseteq \mathbb R\) are intervals and the group operation is given by
\[\textstyle(-\frac 1{2n}, \frac 1{2n}) \times (-\frac 1{2n}, \frac 1{2n}) \to (\frac 1n, \frac 1n), (a, b) \mapsto a + b,\]
is not isomorphic to any pro-group. Moreover, any morphism of group objects from a pro-group \(G^{(\infty)}\) to \(I^{(\infty)}\) is trivial.

The category of pro-sets \(\Pro(\Set)\) has all finite limits, so we may consider algebraic objects in \(\Pro(\Set)\) such as groups and odd form algebras. Every algebraic formula (say, the commutator) defines a morphism in \(\Pro(\Set)\) from a product of such algebraic objects to an algebraic objects. We denote the variable in such formulas with an upper index \((\infty)\). If \(a^{(\infty)}\) is a variable of such a formula, then \(a^{(\infty)} \in X^{(\infty)}\) means that \(X^{(\infty)}\) is the domain of \(a\). For example, \([g^{(\infty)}, h^{(\infty)}]\) for \(g^{(\infty)}, h^{(\infty)} \in G^{(\infty)}\) means the commutator morphism \(G^{(\infty)} \times G^{(\infty)} \to G^{(\infty)}\) for a group object \(G^{(\infty)}\) in \(\Pro(\Set)\).

If \(G^{(\infty)}\) is a group object in \(\Pro(\Set)\) and there is an action \(H \times G^{(\infty)} \to G^{(\infty)}\) by automorphisms (it is not the same as a homomorphism \(H \to \Aut(G^{(\infty)})\) of abstract groups), then their semi-direct product \(G^{(\infty)} \rtimes H\) is also a group object, as a pro-set it is just \(G^{(\infty)} \times H\). The same is true for actions of abstract odd form rings on odd form rings in \(\Pro(\Set)\).

It is relatively easy do define a localization of an odd form \(K\)-algebra \((R, \Delta)\) with respect to a multuplicative subset \(S\). But we also need a ``dual'' construction, the formal projective limit of homotopes of \((R, \Delta)\). In order to define homotopes of the odd form parameter, we need a new kind of objects.

Let \(K\) be a commutative ring. A \(2\)-step nilpotent \(K\)-module is a pair \((M, M_0)\), where \(M\) is a group with a right \(K^\bullet\)-action (the group operation is denoted by \(\dotplus\) and the action by \(m \cdot k\)), \(M_0\) is a subgroup of \(M\) and a left \(K\)-module, and there is a map \(\tau \colon M \to M_0\) such that
\begin{itemize}
\item \([M, M] \leq M_0\), \([M, M_0] = \dot 0\);
\item \([m \cdot k, m' \cdot k'] = kk' [m, m]\);
\item \(m \cdot (k + k') = m \cdot k \dotplus kk' \tau(m) \dotplus m \cdot k'\);
\item \(m \cdot k = k^2 m\) for \(m \in M_0\).
\end{itemize}
In every \(2\)-step nilpotent \(K\)-module \((M, M_0)\) there are identities \(\tau(m) = m \dotplus m \cdot (-1)\) (in particular, \(\tau(m) = 2m\) for \(m \in M_0\)), \(\tau(m \cdot k) = k^2 \tau(m)\), \(\tau(m \dotplus m') = \tau(m) + [m, m'] + \tau(m')\). Also, \(M / M_0\) is a right \(K\)-module. Morphisms of \(2\)-step nilpotent \(K\)-modules \(f \colon (M, M_0) \to (N, N_0)\) are the homomorphisms of groups \(f \colon M \to N\) mapping \(M_0\) to \(N_0\) and preserving all operations.

An augmented odd form \(K\)-algebra is a triple \((R, \Delta, \mathcal D)\) such that \((R, \Delta)\) is an odd form \(K\)-algebra and \(\mathcal D \leq \Delta\) is an \(R^\bullet\)-subgroup and a left \(K\)-module such that for all \(a \in R\), \(k \in K\), \(v \in \mathcal D\)
\begin{itemize}
\item \(\phi(a) \in \mathcal D\), \(\phi(ka) = k \phi(a)\);
\item \(\pi(v) = 0\), \(v \cdot k = k^2 v\), \(\rho(kv) = k \rho(v)\), \(kv \cdot a = k (v \cdot a)\).
\end{itemize}
Then \((\Delta, \mathcal D)\) is a \(2\)-step nilpotent \(K\)-module and the action of \(\unit(R, \Delta)\) preserves the augmentation. Every odd form \(K\)-algebra \((R, \Delta)\) has the smallest augmentation \(\mathcal D = \phi(R)\).

Let \(S \subseteq K\) be a multiplicative subset. We usually use the index category \(\mathcal S\) in formal projective limits, its objects are elements of \(S\), and \(\mathcal S(s, s') = \{s'' \in S \mid s' = s s''\}\). For every \(K\)-module \(M\) let \(M^{(s)} = \{m^{(s)} \mid m \in M\}\), it is isomorphic to \(M\) as a \(K\)-module. Let also \(M^{(\infty, S)}\) be their formal projective limit under the maps \(m^{(ss')} \mapsto (s'm)^{(s)}\). Clearly, \(M^{(\infty, S)}\) is a \(K\)-module in \(\Pro(\Set)\) and a pro-group under addition. Every bilinear map \(f \colon M \times N \to L\) gives a bilinear morphism of pro-sets \(f \colon M^{(\infty, S)} \times N^{(\infty, S)} \to L^{(\infty, S)}, (m^{(s)}, n^{(s)}) \mapsto (sf(m, n))^{(s)}\), this construction preserves various associativity conditions. For example, if \(R\) is a \(K\)-algebra, then \(R^{(\infty, S)}\) is a non-unital \(K\)-algebra in \(\Pro(\Set)\). It is the formal projective limit of algebras \(R^{(s)}\), they are called homotopes of \(R\).

Take a \(2\)-step nilpotent \(K\)-module \((M, M_0)\). Its localization \((S^{-1} M, S^{-1} M_0)\) is a \(2\)-step nilpotent \(S^{-1} K\)-module, where \(S^{-1} M\) is the factor-set of \(S \times M\) by the following equivalence relation: \((s, m) \sim (s', m')\) if and only if there is \(s'' \in S\) such that \(m \cdot s's'' = m' \cdot ss''\). We write a pair \((s, m)\) as \(m \cdot \frac 1s\). The embedding \(S^{-1} M_0 \to S^{-1} M\) is \(\frac ms \mapsto sm \cdot \frac 1s\), the operations on \(S^{-1} M\) are given by
\begin{itemize}
\item \(m \cdot \frac 1s \dotplus m' \cdot \frac 1{s'} = (m \cdot s' \dotplus m' \cdot s) \cdot \frac 1{ss'}\), \((m \cdot \frac 1s) \cdot \frac k{s'} = (m \cdot k) \cdot \frac 1{ss'}\);
\item \([m \cdot \frac 1s, m' \cdot \frac 1{s'}] = \frac{[m, m']}{ss'}\), \(\tau(m \cdot \frac 1s) = \frac{\tau(m)}{s^2}\).
\end{itemize}

Clearly, \((M, M_0) \to (S^{-1} M, S^{-1} M_0), m \mapsto m \cdot \frac 11\) is a morphism of \(2\)-step nilpotent \(K\)-modules and \(S^{-1} M / S^{-1} M_0 \cong S^{-1}(M / M_0)\) as \(S^{-1} K\)-modules.

Now we define homotopes of \((M, M_0)\). Let \(M^{(s)}\) be the group generated by the group \(M_0^{(s)}\) and the elements \(m^{(s)}\) for all \(m \in M\), the map \(M_0^{(s)} \to M^{(s)}\) is denoted by \(\iota\) (so we distinguish the generators \(\iota(m^{(s)})\) and \(m^{(s)}\) for \(m \in M_0\)). The relations are \(m^{(s)} \dotplus m'^{(s)} = (m \dotplus m')^{(s)}\) for all \(m, m'\) and \(m^{(s)} = \iota((sm)^{(s)})\) for \(m \in M_0\). It is easy to see that the sequence \(M_0^{(s)} \to M^{(s)} \to (M / M_0)^{(s)}\) is short exact, where the left map is \(\iota\) and the right map is \(m^{(s)} \mapsto (m + M_0)^{(s)}, \iota(m^{(s)}) \mapsto 0\). Also \((M^{(s)}, M_0^{(s)})\) is a \(2\)-step nilpotent \(K\)-module with
\begin{itemize}
\item \(m^{(s)} \cdot k = (mk)^{(s)}\), \(\iota({m'}^{(s)}) \cdot k = \iota((k^2 m')^{(s)})\);
\item \(\tau(m^{(s)}) = (s \tau(m))^{(s)}\), \(\tau(\iota({m'}^{(s)})) = (2m')^{(s)}\) as elements of \(M_0^{(s)}\);
\item \([m^{(s)}, {m'}^{(s)}] = (s [m, m])^{(s)}\) as elements of \(M_0^{(s)}\).
\end{itemize}

The pro-group \(M^{(\infty, S)}\) is the formal projective limit of \(M^{(s)}\) under the morphisms \(M^{(ss')} \to M^{(s)}, m^{(ss')} \mapsto (m \cdot s')^{(s)}, \iota(m^{(ss')}) \mapsto \iota((s'm)^{(s)})\) of \(2\)-step nilpotent \(K\)-modules. Hence \(M_0^{(\infty, S)} \to M^{(\infty, S)} \to (M / M_0)^{(\infty, S)}\) is a short exact sequence of pro-groups in the following sense: the left morphism is the kernel of the right one, the right morphism is an epimorphism of pro-sets, and any morphism \(f \colon M^{(\infty)} \to X^{(\infty)}\) of pro-sets factors through \((M / M_0)^{(\infty)}\) if and only if \(f\bigl(m^{(\infty)} \dotplus \iota({m'}^{(\infty)})\bigr) = f(m^{(\infty)})\) for \(m^{(\infty)} \in M^{(\infty)}\), \({m'}^{(\infty)} \in M_0^{(\infty)}\) (this is an equality between morphisms \(M^{(\infty)} \times M_0^{(\infty)} \to X^{(\infty)}\)). The morphisms \((-) \cdot (=) \colon M^{(\infty)} \times K \to M^{(\infty)}\), \([-, =] \colon (M / M_0)^{(\infty)} \times (M / M_0)^{(\infty)} \to M_0^{(\infty)}\), \(\tau \colon M^{(\infty)} \to M_0^{(\infty)}\) of pro-sets satisfying the axioms of \(2\)-step nilpotent \(K\)-modules.

Let \((R, \Delta, \mathcal D)\) be an augmented odd form \(K\)-algebra and \(S \subseteq K\) be a multiplicative subset. We have the triples \((S^{-1} R, S^{-1} \Delta, S^{-1} \mathcal D)\) and \((R^{(s)}, \Delta^{(s)}, \mathcal D^{(s)})\) for \(s \in S\) and \(s = \infty\) constructed using the corresponding operations for the \(2\)-step nilpotent \(K\)-module \((\Delta, \mathcal D)\). It is easy to see that the localization \((S^{-1} R, S^{-1} \Delta, S^{-1} \mathcal D)\) is an augmented odd form \(S^{-1} K\)-algebra with the operations \((u \cdot \frac 1s) \cdot \frac a{s'} = (u \cdot a) \cdot \frac 1{ss'}\), \(\phi(\frac as) = \frac{\phi(a)}s\), \(\pi(u \cdot \frac 1s) = \frac{\pi(u)}s\), and \(\rho(u \cdot \frac 1s) = \frac{\rho(u)}{s^2}\). Also, the odd form \(S^{-1} K\)-algebra \((S^{-1} R, S^{-1} \Delta)\) is independent on \(\mathcal D\).

\begin{lemma}\label{ofa-homotopes}
Each homotope \((R^{(s)}, \Delta^{(s)}, \mathcal D^{(s)})\) for \(s \in S\) is an augmented odd form \(K\)-algebra, the operations are given for \(a \in R\), \(u \in \Delta\), \(v \in \mathcal D\) by
\begin{itemize}
\item \(\inv{a^{(s)}} = \inv a^{(s)}\);
\item \(u^{(s)} \cdot a^{(s)} = (u \cdot sa)^{(s)} \in \Delta^{(s)}\), \(v^{(s)} \cdot a^{(s)} = (v \cdot sa)^{(s)} \in \mathcal D^{(s)}\);
\item \(\phi(a^{(s)}) = \phi(a)^{(s)} \in \mathcal D^{(s)}\), \(\pi(u^{(s)}) = \pi(u)^{(s)}\);
\item \(\rho(u^{(s)}) = (s\rho(u))^{(s)}\), \(\rho(\iota(v^{(s)})) = \rho(v)^{(s)}\).
\end{itemize}
The odd form pro-ring \((R^{(\infty, S)}, \Delta^{(\infty, S)})\) is independent on \(\mathcal D\) up to a canonical isomorphism. There is a well-defined action of the abstract odd form ring \((S^{-1} R \rtimes S^{-1} K, S^{-1} \Delta)\) on this odd form pro-ring given by
\begin{itemize}
\item \(a^{(ss')} \frac b{s'} = (ab)^{(s)}\), \(\frac b{s'} a^{(ss')} = (ba)^{(s)}\);
\item \(u^{(ss'^2)} \cdot \frac b{s'} = (u \cdot s'b)^{(s)} \in \Delta^{(s)}\), \(v^{(ss'^2)} \cdot \frac b{s'} = (v \cdot b)^{(s)} \in \mathcal D^{(s)}\);
\item \((w \cdot \frac 1{s'}) \cdot a^{(ss')} = (w \cdot a)^{(s)}\)
\end{itemize}
for \(a, b \in R\), \(u, w \in \Delta\), and \(v \in D\) in the following sense: every element of \(S^{-1} R\) and \(S^{-1} \Delta\) gives morphisms of pro-sets satisfying the axioms. Moreover, \((R \rtimes K, \Delta)\) acts on each \((R^{(s)}, \Delta^{(s)})\).
\end{lemma}
\begin{proof}
All axioms may be checked by direct calculations. To show that \((R^{(\infty, S)}, \Delta^{(\infty, S)})\) is independent on \(\mathcal D\), consider pro-sets \(\widetilde \Delta^{(s)} = \{u^{(s)} \mid u \in \Delta\}\) and their formal projective limit \(\widetilde \Delta^{(\infty, S)}\) under the maps
\[\widetilde \Delta^{(ss')} \to \widetilde \Delta^{(s)}, u^{(ss')} \mapsto (u \cdot s')^{(s)}.\]
The operations on \((R^{(\infty, S)}, \widetilde \Delta^{(\infty, S)})\) are given by
\begin{itemize}
\item \(u^{(s)} \dotplus v^{(s)} = (u \dotplus v)^{(s)}\), \(u^{(s)} \cdot a^{(s)} = (u \cdot a)^{(s)}\);
\item \(\rho(u^{(s)}) = (s \rho(u))^{(s)}\), \(\pi(u^{(s)}) = \pi(u)^{(s)}\), \(\phi(a^{(s^2)}) = \phi(a)^{(s)}\).
\end{itemize}
There is an isomorphism \(\widetilde \Delta^{(\infty, S)} \to \Delta^{(\infty, S)}\) of pro-groups given by \(u^{(s)} \mapsto u^{(s)}\), the inverse one is given by \(\iota(v^{(s^2)}) \mapsto v^{(s)}\) and \(u^{(s^2)} \mapsto (u \cdot s)^{(s)}\) for \(v \in \mathcal D\) and \(u \in \Delta\). So we have an isomorphism \((R^{(\infty, S)}, \Delta^{(\infty, S)}) \cong (R^{(\infty, S)}, \widetilde \Delta^{(\infty, S)})\) of odd form pro-rings in \(\Pro(\Set)\), even though \((R^{(s)}, \widetilde \Delta^{(s)})\) are not actual odd form rings.
\end{proof}

The previous lemma does not give a semi-direct product \((R^{(\infty, S)} \rtimes S^{-1} R \rtimes S^{-1} K, \Delta^{(\infty, S)} \rtimes S^{-1} \Delta)\), since the upper indices in the formulas of this action depend on the denominators. Instead we have morphisms \(R^{(\infty, S)} \times \frac Rs \to R^{(\infty, S)}\), \(\frac Rs \times R^{(\infty, S)} \to R^{(\infty, S)}\), \(\Delta^{(\infty, S)} \times \frac Rs \to \Delta^{(\infty, S)}\), \((\Delta \cdot \frac 1s) \times R^{(\infty, S)} \to \Delta^{(\infty, S)}\) of pro-sets for every \(s \in S\), where \(\frac Rs = \{\frac as \mid a \in R\}\) and \(\Delta \cdot \frac 1s = \{u \cdot \frac 1s \mid u \in \Delta\}\) are abstract sets. The localizations \(S^{-1} R\) and \(S^{-1} \Delta\) are their direct limits, these morphisms are coherent with each other and satisfy the axioms of action of odd form rings. On the other hand, there is a well-defined odd form pro-ring \((R^{(\infty, S)} \rtimes R \rtimes K, \Delta^{(\infty, S)} \rtimes \Delta)\).

\section{Hyperbolic pairs}

From now on let \((R, \Delta, \mathcal D)\) be an odd form \(K\)-algebra. A tuple \(\eta = (e_{-\eta}, e_\eta, q_{-\eta}, q_\eta)\) is called a hyperbolic pair if \(e_\eta, e_{-\eta}\) are orthogonal idempotents in \(R\) with the property \(e_{-\eta} = \inv{e_\eta}\) and \(q_\eta, q_{-\eta}\) are elements of \(\Delta\) such that \(\pi(q_\eta) = e_\eta\), \(\pi(q_{-\eta}) = e_{-\eta}\), \(\rho(q_\eta) = \rho(q_{-\eta}) = 0\), \(q_\eta = q_\eta \cdot e_\eta\), \(q_{-\eta} = q_{-\eta} \cdot e_{-\eta}\). We denote the idempotent \(e_\eta + e_{-\eta}\) by \(e_{|\eta|}\). Hyperbolic pairs \(\eta, \widetilde \eta\) are called orthogonal if \(e_{|\eta|}\) and \(e_{|\widetilde\eta|}\) are orthogonal idempotents, in this case 
\[\eta \oplus \widetilde \eta = (e_{-\eta} + e_{-\widetilde \eta}, e_\eta + e_{\widetilde\eta}, q_{-\eta} \dotplus q_{-\widetilde\eta}, q_\eta \dotplus q_{\widetilde\eta})\]
is a hyperbolic pair (note that \(q_\eta\) and \(q_{-\eta}\) commute with \(q_{\widetilde\eta}\) and \(q_{-\widetilde\eta}\)). Clearly, for any hyperbolic pair \(\eta\) the tuple \(-\eta = (e_\eta, e_{-\eta}, q_\eta, q_{-\eta})\) is also a hyperbolic pair. If \((R, \Delta)\) is constructed from a quadratic module \(M\), then hyperbolic pairs correspond to the so-called hyperbolic subspaces of \(M\) and orthogonal hyperbolic pairs correspond to orthogonal hyperbolic subspaces, see \cite{StabVor} for details.

Let \(e\) and \(\widetilde e\) be idempotents in a non-unital ring \(S\). We say that \(e\) and \(\widetilde e\) are Morita equivalent if \(e \in S\widetilde e S\) and \(\widetilde e \in SeS\) (then the unital subrings \(eSe\) and \(\widetilde e S \widetilde e\) are actually Morita equivalent). Hyperbolic pairs \(\eta\) and \(\widetilde \eta\) in \((R, \Delta, \mathcal D)\) are said to be Morita equivalent if the idempotents \(e_{|\eta|}\), \(e_{|\widetilde \eta|}\) are Morita equivalent in \(R\). For example, \(\eta\) and \(-\eta\) are Morita equivalent. If \(\eta\) and \(\widetilde \eta\) are orthogonal Morita equivalent hyperbolic pairs, then they both are Morita equivalent to \(\eta \oplus \widetilde \eta\).

For any hyperbolic pairs \(\eta\) and \(\widetilde \eta\) we use the notation \(R_{\eta \widetilde \eta} = e_\eta R e_{\widetilde \eta}\), \(\Delta_\eta = \Delta \cdot e_\eta\), \(\Delta^{|\eta|'}_\eta = \{u \in \Delta \cdot e_\eta \mid e_{|\eta|} \pi(u) = 0\}\), and \(\mathcal D_\eta = \mathcal D \cdot e_\eta\). Note that \((\Delta^{|\eta|'}_\eta, \mathcal D_\eta)\) is a \(2\)-step nilpotent \(K\)-module.

Let \(\eta\) be a hyperbolic pair. We associate to \(\eta\) several subgroups of \(\unit(R, \Delta)\). For any element \(u \in \Delta_\eta^{|\eta|'}\) a transvection \(T^\eta(u) \in \unit(R, \Delta)\) is given by
\[\beta(T^\eta(u)) = \rho(u) + \pi(u) - \inv{\pi(u)}, \quad \gamma(T^\eta(u)) = u \dotplus q_{-\eta} \cdot (\rho(u) - \inv{\pi(u)}) \dotminus \phi(\rho(u) + \pi(u)).\]
It is easy to see that \(T^\eta \colon \Delta_\eta^{|\eta|'} \to \unit(R, \Delta)\) is a well-defined injective group homomorphism. Its image is denoted by \(T^\eta(*)\). The elements \(T^\eta(u)\) are similar to the ESD-transvections from \cite{OddDefPetrov}.

For any element \(a \in (R_{\eta\eta})^*\) an elementary dilation \(D^\eta(a) \in \unit(R, \Delta)\) is given by
\[\beta(D^\eta(a)) = a + \inv a^{-1} - e_{|\eta|},\quad \gamma(D^\eta(a)) = q_{-\eta} \cdot (\inv a^{-1} - e_{-\eta}) \dotplus q_\eta \cdot (a - e_\eta) \dotminus \phi(a - e_\eta),\]
where \(a^{-1}\) is the inverse of \(a\) in \(R_{\eta\eta}\) and \(\inv a^{-1} = \inv{a^{-1}}\) is the inverse of \(\inv a\) in \(R_{-\eta, -\eta}\). Again, the map \(D^\eta \colon (R_{\eta\eta})^* \to \unit(R, \Delta)\) is a well-defined injective group homomorphism, its image is denoted by \(D^\eta(*)\). Clearly, \(D^{-\eta}(a) = D^\eta(\inv a^{-1})\).

Note that \((R_{|\eta|' |\eta|'}, \Delta^{|\eta|'}_{|\eta|'}, \mathcal D_{|\eta|'})\) is an augmented odd form \(K\)-subalgebra of \((R, \Delta, \mathcal D)\), where \(R_{|\eta|' |\eta|'} = \{a \in R \mid e_{|\eta|} a = 0 = a e_{|\eta|}\}\), \(\Delta^{|\eta|'}_{|\eta|'} = \{u \in \Delta \mid u \cdot e_{|\eta|} = \dot 0, e_{|\eta|} \pi(u) = 0\}\), and \(\mathcal D_{|\eta|'} = \{u \in \mathcal D \mid u \cdot e_{|\eta|} = 0\}\). Let
\[\unit_{|\eta|'} = \unit(R_{|\eta|' |\eta|'}, \Delta^{|\eta|'}_{|\eta|'}) = \{g \in \unit(R, \Delta) \mid \beta(g) e_{|\eta|} = 0 = e_{|\eta|} \beta(g), \gamma(g) \cdot e_{|\eta|} = \dot 0\},\]
it is a smaller unitary group.

A maximal parabolic subgroup of \(\eta\) is the group
\[P_\eta = \{g \in \unit(R, \Delta) \mid \beta(g) e_{-\eta} = e_{-\eta} \beta(g) e_{-\eta}, e_\eta \beta(g) = e_\eta \beta(g) e_\eta, \gamma(g) \cdot e_{-\eta} = q_{-\eta} \cdot \beta(g) e_{-\eta}\}.\]
Finally, a Levi subgroup of \(P_\eta\) is \(L_\eta = P_\eta \cap P_{-\eta}\).

\begin{lemma}\label{maximal-parabolic}
There are decompositions \(L_\eta = D^\eta(*) \times \unit_{|\eta|'}\) and \(P_\eta = T^\eta(*) \rtimes L_\eta\) (so \(T^\eta(*)\) is called a unipotent radical of \(P_\eta\)). There are identities
\[\up g{(T^\eta(u))} = T^\eta(\gamma(g) \cdot \pi(u) \dotplus u) \text{ for } g \in \unit_{|\eta|'}\]
and
\[T^\eta(u)^{D^\eta(a)} = T^\eta(u \cdot a).\]
\end{lemma}
\begin{proof}
The first decomposition and the identities are easy to check. Since \(T^\eta(*) \cap L_\eta = 1\), there is an inclusion \(T^\eta(*) \rtimes L_\eta \leq P_\eta\). There are well-defined group homomorphisms
\[p_1 \colon P_\eta \to (R_{\eta\eta})^*, g \mapsto e_\eta + e_\eta \beta(g) e_\eta\]
and
\[p_2 \colon P_\eta \to \unit_{|\eta|'}, \beta(p_2(g)) = (1 - e_{|\eta|}) \beta(g) (1 - e_{|\eta|}), \gamma(p_2(g)) = (\gamma(g) \dotminus q_{-\eta} \cdot \beta(g)) \cdot (1 - e_{|\eta|}).\]
From a direct computation it follows that \((p_1, p_2) \colon P_\eta \to L_\eta\) is a retraction of the inclusion \(L_\eta \leq P_\eta\) and its kernel equals to \(T^\eta(*)\). 
\end{proof}

\begin{lemma}\label{parabolic-extraction}
Suppose that \(g \in \unit(R, \Delta)\) satisfies \(a = e_\eta + e_\eta \beta(g) e_\eta \in R_{\eta\eta}^*\). Then there is unique \(u \in \Delta^{|\eta|'}_\eta\) such that \(T^\eta(\dotminus u)\, g \in P_{-\eta}\).
\end{lemma}
\begin{proof}
Indeed, direct calculation shows existence and uniqueness for \(u\). More explicitly,
\begin{align*}
\pi(u) &= (1 - e_{|\eta|}) \beta(g) a^{-1},\\
\rho(u) &= e_{-\eta} \beta(g) a^{-1},\\
u &= (\gamma(g) \dotminus q_{-\eta} \cdot \beta(g) \dotminus q_\eta \cdot \beta(g) \dotplus \phi(\beta(g)) \cdot a^{-1}.\qedhere
\end{align*}
\end{proof}

Now consider two special types of transvections. If \(\eta\) is a hyperbolic pair and \(u \in \mathcal D_\eta\), then
\[\beta(T^\eta(u)) = \rho(u),\enskip \gamma(T^\eta(u)) = q_{-\eta} \cdot \rho(u) \dotminus u.\]
If \(\eta\), \(\widetilde \eta\) are orthogonal hyperbolic pairs and \(a \in R_{\widetilde \eta \eta}\), then
\[\beta(T^\eta(q_{\widetilde \eta} \cdot a)) = a - \inv a,\enskip \gamma(T^\eta(q_{\widetilde \eta} \cdot a)) = q_{\widetilde \eta} \cdot a \dotminus q_{-\eta} \cdot \inv a \dotminus \phi(a).\]
In this case \(T^\eta(q_{\widetilde \eta} \cdot a) = T^{-\widetilde \eta}(q_{-\eta} \cdot (-\inv a))\) for all \(a \in R_{\widetilde \eta \eta}\) and all elements from the intersection \(T^\eta(*) \cap T^{-\widetilde \eta}(*)\) are of this form.

An orthogonal hyperbolic family of rank \(n\) in \((R, \Delta, \mathcal D)\) is a family \(\eta_1, \ldots, \eta_n\) of orthogonal Morita equivalent hyperbolic pairs. In this case we write \(i\) and \(-i\) as indices for the summands of \((R, \Delta, \mathcal D)\) instead of \(\eta_i\) and \(-\eta_i\). For example, \(R_{-2, 1}\) means \(R_{-\eta_2, \eta_1}\). An orthogonal hyperbolic family is called free if there are elements \(e_{ij} \in R\) for \(1 \leq i, j \leq n\) such that \(e_{ij} e_{jk} = e_{ik}\) and \(e_{ii} = e_i\), such families arise in matrix odd form algebras. We say that the orthogonal hyperbolic family is strong if all \(e_i\) for \(-n \leq i \leq n\) and \(i \neq 0\) are pairwise Morita equivalent, this is the case for the odd form algebras constructed from \(K\)-modules with split quadratic or symplectic forms.

From now on fix an orthogonal hyperbolic family \(\eta_1, \ldots, \eta_n\). Elementary transvections are \(T_{ij}(a) = T^{\eta_j}(q_i \cdot a)\) for \(a \in R_{ij}\) and \(T_i(u) = T^{\eta_i}(u)\) for \(u \in \Delta^0_i = \{u \in \Delta_i \mid (e_{|1|} + \ldots + e_{|n|}) \pi(u) = 0\}\). It is easy to see that
\[T^{\eta_i}(*) = T_i(*) \dotoplus \bigoplus^\cdot_{j \neq \pm i} T_{ji}(*),\]
where \(T_i(*)\) and \(T_{ji}(*)\) are the subgroups of all elementary transvections with given indices. An elementary unitary group \(\eunit(R, \Delta)\) is the subgroup of \(\unit(R, \Delta)\) generated by all elementary transvections, i.e.
\[\eunit(R, \Delta) = \langle T_{ij}(*), T_k(*) \rangle.\]

Elementary dilations are \(D_i(a) = D^{\eta_i}(a)\) for \(a \in (R_{ii})^*\) and \(D_0(g) = g\) for \(g \in \unit_{|\eta_1 \oplus \ldots \oplus \eta_n|'}\). Clearly, they satisfy
\begin{enumerate}[label = (D\arabic*), start = 0]
\item \(D_i(a) = D_{-i}(\inv a^{-1})\) for \(i \neq 0\);
\item \(D_i(a)\, D_i(b) = D_i(ab)\) for \(i \neq 0\);
\item \(D_0(g)\, D_0(h) = D_0(gh)\);
\item \([D_i(a), D_j(b)] = 1\) for \(0 \neq i \neq \pm j \neq 0\);
\item \([D_i(a), D_0(g)] = 1\) for \(i \neq 0\).
\end{enumerate}
The product \(\diag(R, \Delta) = \prod_{i = 0}^n D_i(*)\) is the diagonal subgroup of \(\unit(R, \Delta)\), its is the abstract group with the generators \(D_i(a)\) for \(-n \leq i \leq n\) and the relations (D0) -- (D4).

Let us explicitly construct an analog of Pierce decomposition for \((R, \Delta, \mathcal D)\) and its homotopes. First of all, consider any \(2\)-step nilpotent \(K\)-modules \((M_i, M_{i0})\) for \(1 \leq i \leq n\). Fix some bilinear maps \([-, =] \colon M_i / M_{i0} \times M_j / M_{j0} \to \bigoplus_k M_{k0}\) for \(i < j\). Then \((N, N_0) = \bigoplus_i^\cdot (M_i, M_{i0})\) is also \(2\)-step \(K\)-module, where \(N_0 = \bigoplus_i M_{i0}\) and \(N = \bigoplus_i^\cdot M_i\), the operations are defined in the obvious way. Moreover, \(N^{(\infty)} \cong \bigoplus_i^\cdot N_i^{(\infty)}\).

Now there is a decomposition \(R = \bigoplus_{-n \leq i, j \leq n} R_{ij}\), where \(R_{ij} = e_i R e_j\) and \(e_0 = 1 - \sum_{i \neq 0} e_i \in R \rtimes K\). Also \((\Delta, \mathcal D) = \bigoplus_i^\cdot (\Delta_i, \mathcal D_i) \dotoplus \bigoplus_{i + j > 0}^\cdot (\phi(R_{ij}), \phi(R_{ij}))\), where \(\mathcal D_i = \mathcal D \cdot e_i\) and \(\Delta_i = \Delta \cdot e_i\). Note that \(\phi(R_{ij}) \cong R_{ij}\) for \(i + j > 0\). Every \(2\)-step nilpotent \(K\)-module \((\Delta_i, \mathcal D_i)\) further decomposes as \(\bigoplus_{j \neq 0}^\cdot (q_j \cdot R_{ji}, 0) \dotoplus (\Delta_i^0, \mathcal D_i)\), where \(\Delta_i^0 = \{u \in \Delta_i \mid \pi(u) = e_0 \pi(u)\}\). By lemma \ref{ofa-homotopes}, for any \(s \in K\) the same decomposition holds for \((R^{(s)}, \Delta^{(s)}, \mathcal D^{(s)})\):
\begin{align*}
R^{(s)} &= \bigoplus_{-n \leq i, j \leq n} R^{(s)}_{ij},\\
(\Delta^{(s)}, \mathcal D^{(s)}) &= \bigoplus_i^\cdot (\Delta^{0, (s)}_i, \mathcal D^{(s)}_i) \dotoplus \bigoplus_{i \neq 0}^\cdot \bigoplus_j^\cdot (q_i \cdot R_{ij}^{(s)}, 0) \dotoplus \bigoplus_{i + j > 0}^\cdot (\phi(R_{ij}^{(s)}), \phi(R_{ij}^{(s)})),
\end{align*}
where \(R^{(s)}_{ij} = e_i R^{(s)} e_j\), \(\Delta^{0, (s)}_i = \{u^{(s)} \in \Delta^{(s)} \cdot e_i \mid \pi(u^{(s)}) = e_0 \pi(u^{(s)})\}\), \(\mathcal D^{(s)}_i = \mathcal D^{(s)} \cdot e_i\). Also if \(S \subseteq K\) is a multiplicative subset, then we have this decomposition for \((R^{(\infty, S)}, \Delta^{(\infty, S)}, \mathcal D^{(\infty, S)})\).

\section{Steinberg groups}

An odd unitary Steinberg group \(\stunit(R, \Delta)\) is generated by symbols \(X_{ij}(a)\) for \(0 < |i|, |j| \leq n\), \(i \neq \pm j\), \(a \in R_{ij}\) and by symbols \(X_i(u)\) for \(0 < |i| \leq n\), \(u \in \Delta^0_i\). The relations on these symbols are the following:
\begin{enumerate}[label = (St\arabic*), start = 0]
\item \(X_{ij}(a) = X_{-j, -i}(-\inv a)\);
\item \(X_{ij}(a)\, X_{ij}(b) = X_{ij}(a + b)\);
\item \(X_i(u)\, X_i(v) = X_i(u \dotplus v)\);
\item \([X_{ij}(a), X_{kl}(b)] = 1\) for \(i \neq l \neq -j \neq -k \neq i\);
\item \([X_{ij}(a), X_{jk}(b)] = X_{ik}(ab)\) for \(i \neq \pm k\);
\item \([X_{ij}(a), X_{j, -i}(b)] = X_{-i}(\phi(ab))\);
\item \([X_i(u), X_j(v)] = X_{-i, j}(-\inv{\pi(u)} \pi(v))\) for \(i \neq \pm j\);
\item \([X_i(u), X_{jk}(a)] = 1\) for \(j \neq i \neq -k\);
\item \([X_i(u), X_{ij}(a)] = X_{-i, j}(\rho(u) a)\, X_j(\dotminus u \cdot (-a))\).
\end{enumerate}

There is a canonical map \(\stmap \colon \stunit(R, \Delta) \to \unit(R, \Delta)\) given by \(X_i(u) \mapsto T_i(u)\) and \(X_{ij}(a) \mapsto T_{ij}(a)\). Indeed, the elementary transvections satisfy the relations (St0)--(St8) by lemma \(3\) from \cite{StabVor} or by direct computations using our lemma \ref{maximal-parabolic}. The image of \(\stmap\) is the elementary subgroup \(\eunit(R, \Delta)\), its kernel is denoted by \(\kunit_2(R, \Delta)\), and its cokernel is denoted by \(\kunit_1(R, \Delta)\) (it is the set of cosets in general).

The subgroups \(X_{ij}(R_{ij})\), \(X_j(\Delta^0_j)\), and \(X_j(\mathcal D_j)\) of \(\stunit(R, \Delta)\) are called root subgroups. Since \(\stmap\) maps Steinberg generators to elementary transvections, it follows that \(X_{ij}\) and \(X_j\) are injective group homomorphisms. Note that \(X_j(\mathcal D_j)\) depend on the augmentation, unlike the whole Steinberg group.

Let \(\Phi = \{\pm \mathrm e_i \pm \mathrm e_j, \pm \mathrm e_k, \pm 2 \mathrm e_k \mid 1 \leq i < j \leq n, 1 \leq k \leq n\} \subseteq \mathbb R^n\), it is a non-reduced crystallographic root system of type \(\mathsf{BC}_n\). Its elements are called long roots, short roots and ultrashort roots depending on their length (\(2\), \(\sqrt 2\), or \(1\)). Let also \(\mathrm e_{-i} = -\mathrm e_i\) for \(1 \leq i \leq n\). For any root \(\alpha\) we assign a root subgroup of \(\stunit(R, \Delta)\) in the following way:
\[
X_\alpha(*) = \begin{cases}
X_{ij}(R_{ij}), & \alpha = \mathrm e_j - \mathrm e_i, i \neq \pm j;\\
X_i(\Delta^0_i), & \alpha = \mathrm e_i;\\
X_i(\mathcal D_i), & \alpha = 2 \mathrm e_i.
\end{cases}
\]
The Steinberg relations imply that this definition is correct for short roots and
\[[X_\alpha(*), X_\beta(*)] \leq \prod_{\substack{i\alpha + j\beta \in \Phi\\ i, j > 0}} X_{i\alpha + j\beta}(*)\]
for all non-antiparallel roots \(\alpha, \beta \in \Phi\), where the right hand side is a nilpotent subgroup of \(\stunit(R, \Delta)\).

Now we a ready to define Steinberg pro-groups for multiplicative subseteq \(S \subseteq K\). Let \(\stunit^{(s)}(R, \Delta) = \stunit(R^{(s)}, \Delta^{(s)})\) be the group generated by symbols \(X_{ij}(a)\) and \(X_j(u)\) for \(a \in R_{ij}^{(s)}\) and \(u \in \Delta^{0, (s)}_j\), where \(s \in S\). The relations on this symbols are (St0)--(St8). There are obvious structure homomorphisms \(\stunit^{(ss')}(R, \Delta) \to \stunit^{(s)}(R, \Delta)\), and a Steinberg pro-group \(\stunit^{(\infty, S)}(R, \Delta)\) is the formal projective limit of all \(\stunit^{(s)}(R, \Delta)\) for \(s \in S\). We prove in the next section that it is independent on the augmentation. The generators may be considered as morphisms \(X_{ij} \colon R^{(\infty, S)}_{ij} \to \stunit^{(\infty, S)}(R, \Delta)\) and \(X_j \colon \Delta^{0, (\infty, S)}_j \to \stunit^{(\infty, S)}(R, \Delta)\) of pro-groups. Also there is a morphism \(\stmap \colon \stunit^{(\infty, S)}(R, \Delta) \to \unit(R^{(\infty, S)}, \Delta^{(\infty, S)})\) of pro-groups, where every \(\stmap \colon \stunit^{(s)}(R, \Delta) \to \unit(R^{(s)}, \Delta^{(s)})\) comes from the composition
\[\stunit^{(s)}(R, \Delta) \to \stunit(R^{(s)} \rtimes R, \Delta^{(s)} \rtimes \Delta) \to \unit(R^{(s)} \rtimes R, \Delta^{(s)} \rtimes \Delta).\]

Note that all constructions from \cite{LinK2} may be stated in terms of odd form algebra. If \(S\) is a \(K\)-algebra with a complete family of orthogonal idempotents \(\eps_1, \ldots, \eps_n\), then \(R = S^\op \times S\) is a \(K\)-algebra with the involution \(\inv{(a^\op, b)} = (b^\op, a)\), \((R, \Delta^{\max})\) is a special unital odd form \(K\)-algebra, and \(\unit(R, \Delta^{\max}) \cong S^*\). Moreover, \(e_{-i} = (\eps_i^\op, 0)\) and \(e_i = (0, \eps_i)\) are orthogonal idempotents for \(0 < i \leq n\), they determine orthogonal hyperbolic pairs \(\eta_i = (e_{-i}, e_i, q_{-i}, q_i)\) (here \(q_{-i}\) and \(q_i\) are uniquely determined since the odd form algebra is special). Also \(\eps_i\) are Morita equivalent if and only if \(\eta_i\) are Morita equivalent. The linear Steinberg group constructed by \(\eps_i\) is canonically isomorphic to the unitary Steinberg group constructed by \(\eta_i\).

The diagonal group \(\diag(R, \Delta)\) acts on \(\stunit(R, \Delta)\) (and on each root subgroup) by
\begin{enumerate}[label = (Ad\arabic*)]
\item \(\up{D_i(a)}{X_{jk}(b)} = X_{jk}(b)\) for \(j \neq \pm i \neq k\);
\item \(\up{D_0(g)}{X_{ij}(a)} = X_{ij}(a)\);
\item \(\up{D_i(a)}{X_{ij}(b)} = X_{ij}(ab)\);
\item \(\up{D_i(a)}{X_j(u)} = X_j(u)\) for \(i \neq \pm j\);
\item \(\up{D_0(g)}{X_i(u)} = X_i(\gamma(g) \cdot \pi(u) \dotplus u)\);
\item \(\up{D_i(a)}{X_i(u)} = X_i(u \cdot a^{-1})\);
\end{enumerate}
where \(a^{-1}\) is the inverse in the ring \(R_{ii}\).

Note that the Weyl group \(W = (\mathbb Z / 2 \mathbb Z)^n \rtimes \mathrm S_n\) of \(\Phi\) acts on the set of root subgroups by permutations of roots (and on the set of \(D_i(*)\) by permutations of indices). It is easy to see that \(W\) acts transitively on all non-zero indices, on all pairs \((i, j)\) of indices with \(0 \neq i \neq \pm j \neq 0\), and on all roots of given length. Also, \(W\) acts transitively on all pairs of non-collinear roots \((\alpha, \beta)\) with given lengths, given angle between them, and given Dynkin diagram of \(\Phi \cap (\mathbb R\alpha + \mathbb R\beta)\). This Dynkin diagram may be \(\mathsf A_1 \times \mathsf A_1\), \(\mathsf A_1 \times \mathsf{BC}_1\), \(\mathsf A_2\), or \(\mathsf{BC}_2\). It is not uniquely determined by the lengths and the angle only in the case of orthogonal short roots (where there is a choice between \(\mathsf{A}_1 \times \mathsf{A}_1\) and \(\mathsf{BC}_2\)).

Similarly to \cite{LinK2}, it is possible to define root systems \(\Phi / \alpha\) and corresponding orthogonal hyperbolic families for all short and ultrashort \(\alpha \in \Phi\). Namely, fix such a root \(\alpha \in \Phi\). The set \(\Phi / \alpha\) is the image of \(\Phi \setminus \mathbb R\alpha\) in \(\mathbb R^n / \mathbb R\alpha\), it consists of classes \([\beta]\) for all roots \(\beta \in \Phi \setminus \mathbb R\alpha\). Let \(\stunit(R, \Delta; \Phi) = \stunit(R, \Delta)\) and \(\diag(R, \Delta; \Phi) = \diag(R, \Delta)\). We construct the new orthogonal hyperbolic family case by case up to the action of the new Weyl group \((\mathbb Z / 2 \mathbb Z)^{n - 1} \rtimes \mathrm S_{n - 1}\).

If \(\alpha\) is short, then without loss of generality \(\alpha = \mathrm e_n - \mathrm e_{n - 1}\). Let \(\eta_\infty = \eta_{n - 1} \oplus \eta_n\) be the new hyperbolic pair, \(\stunit(R, \Delta; \Phi / \alpha)\) and \(\diag(R, \Delta; \Phi / \alpha)\) be the groups with respect to the orthogonal hyperbolic family \(\eta_1, \ldots, \eta_{n - 2}, \eta_\infty\). There is a well-defined homomorphism \(F_\alpha \colon \stunit(R, \Delta; \Phi / \alpha) \to \stunit(R, \Delta; \Phi)\) given by
\begin{itemize}
\item \(X_{ij}(a) \mapsto X_{ij}(a)\) for \(i, j \neq \pm \infty\);
\item \(X_{i \infty}(a) \mapsto X_{i, n - 1}(ae_{n - 1})\, X_{in}(ae_n)\) for \(i \neq \pm \infty\);
\item \(X_{\infty j}(a) \mapsto X_{n - 1, j}(e_{n - 1} a)\, X_{n j}(e_n a)\) for \(j \neq \pm \infty\);
\item \(X_j(u) \mapsto X_j(u)\) for \(j \neq \pm \infty\);
\item \(X_\infty(u) \mapsto X_{n - 1}(u \cdot e_{n - 1})\, X_n(u \cdot e_n)\, X_{-n, n - 1}(e_{-n} \rho(u) e_{n - 1})\);
\item \(X_{-\infty}(u) \mapsto X_{1 - n}(u \cdot e_{1 - n})\, X_{-n}(u \cdot e_{-n})\, X_{n, 1 - n}(e_n \rho(u) e_{1 - n})\).
\end{itemize}
On \(X_{i, -\infty}\) and \(X_{-\infty, i}\) this homomorphism is defined by (St0).

If \(\alpha\) is ultrashort, then without loss of generality \(\alpha = \mathrm e_1\). Let \(\stunit(R, \Delta; \Phi / \alpha)\) and \(\diag(R, \Delta; \Phi / \alpha)\) be the groups with respect to the orthogonal hyperbolic family \(\eta_2, \ldots, \eta_n\). There is a well-defined homomorphism \(F_\alpha \colon \stunit(R, \Delta; \Phi / \alpha) \to \stunit(R, \Delta; \Phi)\) given by
\begin{itemize}
\item \(X_{ij}(a) \mapsto X_{ij}(a)\);
\item \(X_i(u) \mapsto X_i(u \dotminus q_1 \cdot \pi(u) \dotminus q_{-1} \cdot \pi(u))\, X_{-1, i}(e_{-1} \pi(u))\, X_{1, i}(e_1 \pi(u))\).
\end{itemize}
Note that in this case \(X_i\) are defined on different sets in \(\stunit(R, \Delta; \Phi / \alpha)\) and \(\stunit(R, \Delta; \Phi)\).

In all cases \(\stmap_{\Phi} \circ F_\alpha = \stmap_{\Phi / \alpha} \colon \stunit(R, \Delta; \Phi / \alpha) \to \unit(R, \Delta)\). The set \(\Phi / \alpha\) is a root system of type \(\mathsf{BC}_{n - 1}\), it parametrizes the root subgroups of \(\stunit(R, \Delta; \Phi / \alpha)\). We have \(F_\alpha(X_{[\beta]}(*)) = \prod_{\beta + i\alpha \in \Phi} X_{\beta + i\alpha}(*)\) and \(\langle \diag(R, \Delta; \Phi), T_\alpha(*), T_{-\alpha}(*) \rangle \leq \diag(R, \Delta; \Phi / \alpha)\), where \(T_\alpha(*) = \stmap(X_\alpha(*))\). Clearly, \(\Phi / \alpha = \Phi / (-\alpha)\) (i.e. there is a canonical bijection between these sets, the corresponding Steinberg and diagonal groups are also canonically isomorphic) and \((\Phi / \alpha) / [\beta] \cong (\Phi / \beta) / [\alpha]\) for non-collinear \(\alpha\) and \(\beta\) (here both sides are defined or not simultaneously). We denote \((\Phi / \alpha) / [\beta]\) by \(\Phi / \{\alpha, \beta\}\), it is defined if and only if neither \(\alpha\) nor \(\beta\) is long, \(\alpha\) and \(\beta\) are non-collinear, and they are not orthogonal short roots inside some root subsystem of type \(\mathsf{BC}_2\). Note that \(\Phi / \{\alpha, \beta\}\) depends only on the span of \(\alpha\) and \(\beta\). We say that \(\Phi / \alpha\) is obtained from \(\Phi\) by elimination of \(\alpha\), similarly for the orthogonal hyperbolic family and the Steinberg group.

We also need the subgroups
\begin{align*}
U^+(R, \Delta; \Phi) &= \langle X_{ij}(R_{ij}), T_k(\Delta^0_k) \mid i < j \text{ and } k > 0 \rangle,\\
U^-(R, \Delta; \Phi) &= \langle X_{ij}(R_{ij}), T_k(\Delta^0_k) \mid i > j \text{ and } k < 0 \rangle
\end{align*}
of the Steinberg group \(\stunit(R, \Delta; \Phi)\). Clearly, these groups are nilpotent and \(\stmap\) is injective on them. If \(\alpha\) lies in the basis of \(\mathsf{BC}_n\) (i.e. it is \(\mathrm e_1\) or \(\mathrm e_{i + 1} - \mathrm e_i\) for \(1 \leq i \leq n - 1\)), then we may similarly define \(\unit^{\pm}(R, \Delta; \Phi / \alpha)\). It is easy to see that \(U^{\pm}(R, \Delta; \Phi) = F_\alpha(\unit^{\pm}(R, \Delta; \Phi / \alpha)) \rtimes X_{\pm \alpha}(*)\) and \(F_\alpha(U^{\pm}(R, \Delta; \Phi / \alpha))\) are normalized by both \(X_\alpha(*)\) and \(X_{-\alpha}(*)\).

We say that \(P^{\pm}(R, \Delta; \Phi) = \stmap(U^{\pm}(R, \Delta; \Phi)) \rtimes \diag(R, \Delta; \Phi)\) are opposite parabolic subgroups of \(\unit(R, \Delta)\), \(\diag(R, \Delta; \Phi)\) is their common Levi subgroup, and \(U^{\pm}(R, \Delta; \Phi)\) are their unipotent radicals. If the orthogonal hyperbolic family has rank \(1\), then we get the maximal parabolic subgroups from lemma \ref{maximal-parabolic}. Clearly, \(\stunit(R, \Delta; \Phi)\) is generated by \(U^+(R, \Delta; \Phi)\) and \(U^-(R, \Delta; \Phi)\).

\section{Presentations of pro-groups}

We need several simple lemmas about various pro-groups. Let \(S \subseteq K\) be a multiplicative subset.

\begin{lemma}\label{steinberg-presentation}
Let \(G^{(\infty)}\) be a pro-group. Then every morphism \(\stunit^{(\infty, S)}(R, \Delta) \to G^{(\infty)}\) of pro-groups is uniquely determined by its compositions with the generators \(X_{ij}\) and \(X_j\). Morphisms \(f_{ij} \colon R_{ij}^{(\infty, S)} \to G^{(\infty)}\) and \(f_k \colon \Delta^{0, (\infty, S)}_k \to G^{(\infty)}\) of pro-sets are restrictions of a morphism \(\stunit^{(\infty, S)}(R, \Delta) \to G^{(\infty)}\) of pro-groups if and only if they satisfy (St0)--(St8) in \(\Pro(\Set)\).
\end{lemma}
\begin{proof}
This follows directly from the definition of \(\stunit^{(\infty, S)}(R, \Delta)\). Here we use that there is only a finite number of these generators and that \(G^{(\infty)}\) is an actual pro-group, not a group object in \(\Pro(\Set)\).
\end{proof}

The previous lemma shows, in particular, that \(\stunit^{(\infty, S)}(R, \Delta)\) is independent on the augmentation up to a canonical isomorphism. The formulas \((\mathrm{Ad}1)\)--\((\mathrm{Ad}6)\) define a homomorphism \(\diag(S^{-1} R, S^{-1} \Delta) \to \Aut(\stunit^{(\infty, S)}(R, \Delta))\). Also there is a canonical morphism \(F_\alpha \colon \stunit^{(\infty, S)}(R, \Delta; \Phi / \alpha) \to \stunit^{(\infty, S)}(R, \Delta; \Phi)\) of pro-groups for each non-long root \(\alpha \in \Phi\) defined as for the ordinary Steinberg groups.

Recall that the orthogonal hyperbolic family is strong if all \(e_i\) are Morita equivalent. A morphism \(f \colon X \to Y\) in any category is called a split epimorphism if it admits a section (then it is an epimorphism in the usual sense). Split epimorphisms are preserved under pullbacks.

\begin{lemma}\label{ring-generation}
Let \(i, j, k\) be non-zero indices from \(\{-n, \ldots, n\}\). Then there is \(N \geq 0\) such that the morphism 
\[\sum_{p = 1}^N a_p^{(\infty)} b_p^{(\infty)} + \sum_{p = 1}^N c_p^{(\infty)} d_p^{(\infty)} \colon (R_{ij}^{(\infty, S)} \times R_{jk}^{(\infty, S)})^N \times (R_{i, -j}^{(\infty, S)} \times R_{-j, k}^{(\infty, S)})^N \to R_{ik}^{(\infty, S)}\]
is a split epimorphism in \(\Pro(\Set)\). If the orthogonal hyperbolic family is strong, then the same is true for
\[\sum_{p = 1}^N a_p^{(\infty)} b_p^{(\infty)} \colon (R_{ij}^{(\infty, S)} \times R_{jk}^{(\infty, S)})^N \to R_{ik}^{(\infty, S)}.\]
\end{lemma}
\begin{proof}
Let \(e_i = \sum_{l = \pm j} \sum_p x_{lp} y_{lp}\) for some \(x_{lp} \in R_{il}\) and \(y_{lp} \in R_{li}\), such elements exist since \(\eta_i\) and \(\eta_p\) are Morita equivalent. Then the maps \(a^{(s^2)} \mapsto \bigl(x_{lp}^{(s)}, y_{lp} a^{(s)}\bigr)_{l, p}\) give a required section. The proof of the second claim is similar.
\end{proof}

\begin{lemma}\label{form-generation}
Let \(i, j\) be non-zero indices from \(\{-n, \ldots, n\}\). Then there is \(N \geq 0\) such that the morphism
\begin{align*}
\sum_{1 \leq p \leq N}^\cdot u_p^{(\infty)} \cdot a_p^{(\infty)} &\dotplus \sum_{1 \leq p \leq N}^\cdot v_p^{(\infty)} \cdot b_p^{(\infty)} \dotplus \phi(c^{(\infty)}) \colon\\
&\bigl(\Delta^{0, (\infty, S)}_j \times R_{ji}^{(\infty, S)}\bigr)^N \times \bigl(\Delta^{0, (\infty, S)}_{-j} \times R_{-j, i}^{(\infty, S)}\bigr)^N \times R_{-i, i}^{(\infty, S)} \to \Delta^{0, (\infty, S)}_i
\end{align*}
is a split epimorphism in \(\Pro(\Set)\). If the orthogonal hyperbolic family is strong, then the same is true for
\[\sum_{1 \leq p \leq N}^\cdot u_p^{(\infty)} \cdot a_p^{(\infty)} \dotplus \phi(b^{(\infty)}) \colon \bigl(\Delta^{0, (\infty, S)}_j \times R_{ji}^{(\infty, S)}\bigr)^N \times R_{-i, i}^{(\infty, S)} \to \Delta^{0, (\infty, S)}_i.\]
\end{lemma}
\begin{proof}
Let \(e_i = \sum_{l = \pm j} \sum_p x_{lp} y_{lp}\) for some \(x_{lp} \in R_{il}\) and \(y_{lp} \in R_{li}\). Then the maps
\begin{align*}
u^{(s^3)} \dotplus \iota(v^{(s^3)}) &\mapsto \Bigl(\bigl(u^{(s)} \cdot sx_{lp} \dotplus \iota(v^{(s)}) \cdot x_{lp}, y_{lp}^{(s)}\bigr)_{l, p},\\
& \quad \sum_{(l, p) < (l', p')} \inv{y_{l'p'}} \inv{x_{l'p'}}\, (s^5 \rho(u) + s^2 \rho(v))^{(s)}\, x_{lp} y_{lp} \Bigr)
\end{align*}
give a required section. The strong case is similar.
\end{proof}

\begin{lemma}\label{perfect}
If \(n \geq 3\), then the Steinberg group \(\stunit(R, \Delta)\) is perfect.
\end{lemma}
\begin{proof}
Indeed, let \(0 \neq i \neq k \neq 0\) and choose an index \(j\) different from \(0\), \(\pm i\), \(\pm j\). Then \(X_{ik}(*)\) lies in the subgroup generated by \([X_{i, \pm k}(*), X_{\pm k, j}(*)]\) by (St1), (St4), and lemma \ref{ring-generation} applied to \(S = \{1\}\). Similarly, let \(0 \neq i\) and choose an index \(j\) different from \(0\) and \(\pm i\). Then \(X_i(*)\) lies in the subgroup generated by \([T_{\pm j}(*), T_{\pm j, i}(*)]\), \(T_{\pm j, i}(*)\), \([T_{-i, \pm j}(*), T_{\pm j, i}(*)]\) by (St2), (St5), (St8), and lemmas \ref{ring-generation}, \ref{form-generation} applied to \(S = \{1\}\).
\end{proof}

Lemmas \ref{ring-generation} and \ref{form-generation} say that \(R^{(\infty, S)}_{ij}\) and \(\Delta^{0, (\infty, S)}_j\) are ``generated'' by the same objects with other indices. Now we give the presentations of these objects.

\begin{lemma}\label{ring-presentation}
Let \(i, j, k\) be non-zero indices from \(\{-n, \ldots, n\}\), \(G^{(\infty)}\) be a pro-group, \(f_{\pm j} \colon R_{i, \pm j}^{(\infty, S)} \times R_{\pm j, k}^{(\infty, S)} \to G^{(\infty)}\) be morphisms of pro-sets. Then \(f_{\pm j}\) factor through a morphism \(g \colon R_{ik}^{(\infty, S)} \to G^{(\infty)}\) of pro-groups if and only if they satisfy
\begin{itemize}
\item \(\bigl[f_l\bigl(a^{(\infty)}, b^{(\infty)}\bigr), f_{l'}\bigl({a'}^{(\infty)}, {b'}^{(\infty)}\bigr)\bigr] = 1\) for \(l, l' = \pm j\);
\item \(f_l\bigl(a^{(\infty)} + {a'}^{(\infty)}, b^{(\infty)}\bigr) = f_l\bigl(a^{(\infty)}, b^{(\infty)}\bigr)\, f_l\bigl({a'}^{(\infty)}, b^{(\infty)}\bigr)\) for \(l = \pm j\);
\item \(f_l\bigl(a^{(\infty)}, b^{(\infty)} + {b'}^{(\infty)}\bigr) = f_l\big(a^{(\infty)}, b^{(\infty)}\bigr)\, f_l\bigl(a^{(\infty)}, {b'}^{(\infty)}\bigr)\) for \(l = \pm j\);
\item \(f_l\bigl(a^{(\infty)}, b^{(\infty)}c^{(\infty)}\bigr) = f_{l'}\bigl(a^{(\infty)}b^{(\infty)}, c^{(\infty)}\bigr)\) for \(l, l' = \pm j\), \(a^{(\infty)} \in R_{il}^{(\infty, S)}\), \(b^{(\infty)} \in R_{ll'}^{(\infty, S)}\), \(c^{(\infty)} \in R_{l'k}^{(\infty, S)}\).
\end{itemize}
If the orthogonal hyperbolic family is strong, then the same is true only for \(f_j\) without \(f_{-j}\) and for \(l = l' = j\) in all the formulas.
\end{lemma}
\begin{proof}
Suppose that \(f_{\pm j}\) satisfy the formulas (in the other direction the claim is obvious). Note that \(g\) is unique by lemma \ref{ring-generation}, hence we may assume that \(G\) is a group. Let \(e_i = \sum_{l = \pm j} \sum_p x_{lp} y_{lp}\) for some \(x_{lp} \in R_{il}\) and \(y_{lp} \in R_{li}\). Define \(g\) by
\[g\bigl(c^{(s^2)}\bigr) = \prod_{l = \pm j} \prod_p f_l\bigl(x_{lp}^{(s)}, y_{lp} c^{(s)}\bigr)\]
for sufficiently large \(s\). It is easy to see that \(g\) is a homomorphism for large \(s\). Moreover,
\begin{align*}
f_{\pm j}\bigl(a^{(s^2)}, b^{(s^2)}) &= \prod_{l, p} f_{\pm j}\bigl(s x_{lp} y_{lp} a^{(s)}, sb^{(s)}\bigr)\\
&= \prod_{l, p} f_{\pm j}\bigl(x_{lp}^{(s)}, (s^2 y_{lp} a b)^{(s)}\bigr) = g\bigl((s^2 ab)^{(s^2)}\bigr)
\end{align*}
for large \(s\). In the strong case the proof is similar.
\end{proof}

\begin{lemma}\label{form-presentation}
Let \(i, j\) be non-zero indices from \(\{-n, \ldots, n\}\), \(G^{(\infty)}\) be a pro-group, \(f_{\pm j} \colon \Delta^{0, (\infty, S)}_{\pm j} \times R_{\pm j, i}^{(\infty, S)} \to G^{(\infty)}\) and \(g \colon R_{-i, i}^{(\infty, S)} \to G^{(\infty)}\) be morphisms of pro-sets. Then \(f_{\pm j}\) and \(g\) factor through a morphism \(h \colon \Delta^{0, (\infty, S)}_i \to G^{(\infty)}\) of pro-groups if and only if they satisfy
\begin{itemize}
\item \(\bigl[f_l\bigl(u^{(\infty)}, a^{(\infty)}\bigr), g\bigl(b^{(\infty)}\bigr)\bigr] = 1\) for \(l = \pm j\);
\item \(\bigl[f_l\bigl(u^{(\infty)}, a^{(\infty)}\bigr), f_{l'}\bigl(v^{(\infty)}, b^{(\infty)}\bigr)\bigr] = g\bigl(-\inv{a^{(\infty)}} \inv{\pi(u^{(\infty)})} \pi(v^{(\infty)}) b^{(\infty)}\bigr)\) for \(l, l' = \pm j\);
\item \(f_l\bigl(u^{(\infty)} \dotplus {u'}^{(\infty)}, a^{(\infty)}\bigr) = f_l\bigl(u^{(\infty)}, a^{(\infty)}\bigr)\, f_l\bigl({u'}^{(\infty)}, a^{(\infty)}\bigr)\) for \(l = \pm j\);
\item \(f_l\bigl(u^{(\infty)}, a^{(\infty)} + {a'}^{(\infty)}\bigr) = f_l\bigl(u^{(\infty)}, a^{(\infty)}\bigr)\, g\bigl(\inv{{a'}^{(\infty)}} \rho(u^{(\infty)}) a^{(\infty)}\bigr)\, f_l\bigl(u^{(\infty)}, {a'}^{(\infty)}\bigr)\) for \(l = \pm j\);
\item \(f_l\bigl(u^{(\infty)}, a^{(\infty)} b^{(\infty)}\bigr) = f_{l'}\bigl(u^{(\infty)} \cdot a^{(\infty)}, b^{(\infty)}\bigr)\) for \(l, l' = \pm j\), \(u^{(\infty)} \in \Delta^{0, (\infty)}_l\), \(a^{(\infty)} \in R^{(\infty)}_{ll'}\), \(b^{(\infty)} \in R^{(\infty)}_{l'i}\);
\item \(g\bigl(a^{(\infty)} + {a'}^{(\infty)}\bigr) = g\bigl(a^{(\infty)}\bigr)\, g\bigl({a'}^{(\infty)}\bigr)\);
\item \(g\bigl(\inv{a^{(\infty)}}\bigr) = g\bigl(a^{(\infty)}\bigr)^{-1}\);
\item \(f_l\bigl(\phi(a^{(\infty)}), b^{(\infty)}\bigr) = g\bigl(\inv{b^{(\infty)}} a^{(\infty)} b^{(\infty)}\bigr)\) for \(l = \pm j\), \(a^{(\infty)} \in R_{-l, l}^{(\infty)}\), \(b^{(\infty)} \in R^{(\infty)}_{li}\).
\end{itemize}
If the orthogonal hyperbolic family is strong, then the same is true without \(f_{-j}\) and for \(l = l' = j\) in all the formulas.
\end{lemma}
\begin{proof}
Suppose that \(f_{\pm j}\) and \(g\) satisfy the formulas (in the other direction the claim is obvious). Note that \(h\) is uniquely determined by lemma \ref{form-generation}, hence we may assume that \(G\) is a group. Let \(e_i = \sum_{l = \pm j} \sum_p x_{lp} y_{lp}\) for some \(x_{lp} \in R_{il}\) and \(y_{lp} \in R_{li}\). Define \(h\) by
\begin{align*}
h\bigl(u^{(s^3)} \dotplus \iota(v^{(s^3)})\bigr) &= \prod_{l, p} f_l\bigl(u^{(s)} \cdot sx_{lp} \dotplus \iota(v^{(s)}) \cdot x_{lp}, y_{lp}^{(s)}\bigr)\\
&\quad g\Bigl(\sum_{(l, p) < (l', p')} \inv{y_{l'p'}} \inv{x_{l'p'}}\, (s^5 \rho(u) + s^2 \rho(v))^{(s)}\, x_{lp} y_{lp}\Bigr)
\end{align*}
for sufficiently large \(s\), where \((l, p)\) runs over \(\{j, -j\} \times \{1, \ldots, N\}\) at the lexicographic order. It is easy to see that \(h\) is a homomorphism for large \(s\). Moreover,
\begin{align*}
f_{\pm j}\bigl(u^{(s^3)} \dotplus \iota(v^{(s^3)}), a^{(s^3)}\bigr) &= \prod_{l, p} f_{\pm j} \bigl(u^{(s^3)} \dotplus \iota(v^{(s^3)}), (a x_{lp} y_{lp})^{(s^3)}\bigr)\\
&\quad g\Bigl(\sum_{(l, p) < (l', p')} \inv{(a x_{l'p'} y_{l'p'})^{(s^3)}}\, \rho(u^{(s^3)} \dotplus \iota(v^{(s^3)}))\, (a x_{lp} y_{lp})^{(s^3)}\Bigr)\\
&= \prod_{l, p} f_l\bigl((u^{(s^3)} \dotplus \iota(v^{(s^3)})) \cdot (s a x_{lp})^{(s)}, y^{(s)}_{lp}\bigr)\\
&\quad g\Bigl(\sum_{(l, p) < (l', p')} \inv{(a x_{l'p'} y_{l'p'})^{(s^3)}}\, \rho(u^{(s^3)} \dotplus \iota(v^{(s^3)}))\, (a x_{lp} y_{lp})^{(s^3)}\Bigr)\\
&= h\bigl((u \cdot s^3 a)^{(s^3)} \dotplus \iota((v \cdot s^3 a)^{(s^3)})\bigr) 
\end{align*}
and
\begin{align*}
g\bigl(a^{(s^3)}\bigr) &= g\Bigl(\sum_{(l, p)} \sum_{(l', p')} \inv{(x_{l'p'} y_{l'p'})^{(s)}}\, a^{(s)}\, (x_{lp} y_{lp})^{(s)}\Bigr)\\
&= \prod_{l, p} f_l\bigl(\iota(\phi(\inv{x_{lp}} a x_{lp})^{(s)}), y_{lp}^{(s)}\bigr)\\
&\quad g\Bigl(\sum_{(l, p) < (l', p')} \inv{(x_{l'p'} y_{l'p'})^{(s)}}\, (a - \inv a)^{(s)}\, (x_{lp} y_{lp})^{(s)}\Bigr)\\
&= h\bigl(\iota(\phi(a)^{(s^3)})\bigr)
\end{align*}
for large \(s\). The strong case is similar.
\end{proof}

\section{Elimination of roots: surjectivity}

Let \(S \subseteq K\) be a multiplicative subset. We are ready to prove that the morphism \(F_\alpha \colon \stunit^{(\infty, S)}(R, \Delta; \Phi / \alpha) \to \stunit^{(\infty, S)}(R, \Delta; \Phi)\) of pro-groups is an isomorphism for any \(\alpha\) if \(n\) is sufficiently large.

In the proofs we often implicitly use the group-theoretic identities \([xy, z] = \up x{[y, z]}\, [x, z]\), \([x, yz] = [x, y]\, \up y{[x, z]}\), \(\up x{[y, z]} = [\up xy, \up xz]\), \([x, y] = \up xy\, y^{-1} = x\, (\up yx)^{-1}\), and \([x, y]^{-1} = [y, x]\).

\begin{lemma}\label{root-generation}
Suppose that \(n \geq 3\) and \(\alpha \in \Phi\) is a non-long root. Then \(F_\alpha \colon \stunit^{(\infty, S)}(R, \Delta; \Phi / \alpha) \to \stunit^{(\infty, S)}(R, \Delta; \Phi)\) is an epimorphism of pro-groups. If \(\Phi / \{\alpha, \beta\}\) is defined and in addition the orthogonal hyperbolic family is strong or \(\beta\) is ultrashort, then \(F_{\{\alpha, \beta\}} \colon \stunit^{(\infty, S)}(R, \Delta; \Phi / \{\alpha, \beta\}) \to \stunit^{(\infty, S)}(R, \Delta; \Phi)\) is also an epimorphism of pro-groups.
\end{lemma}
\begin{proof}
To prove the first claim, note that any generator \(X_\gamma\) of \(\stunit^{(\infty, S)}(R, \Delta; \Phi)\) factors through \(F_\alpha\) unless \(\gamma\) and \(\alpha\) are collinear. By symmetry, it suffices to consider only the case \(\gamma = \alpha\). If \(\alpha\) is short, then without loss of generality \(\alpha = \mathrm e_n - \mathrm e_{n - 1}\). Then
\[\bigl[X_{n - 1, \pm 1}\bigl(a^{(\infty)}\bigr), X_{\pm 1, n}\bigl(b^{(\infty)}\bigr)\bigr] = X_{n - 1, n}\bigl(a^{(\infty)} b^{(\infty)}\bigr),\]
hence we are done by lemmas \ref{steinberg-presentation} and \ref{ring-generation}. If \(\alpha\) is ultrashort, then without loss of generality \(\alpha = \mathrm e_1\). Then
\begin{align*}
X_{\mp 2, 1}\bigl(-\inv{\rho(u^{(\infty)})} b^{(\infty)}\bigr)\, \bigl[X_{\pm 2}\bigl(\dotminus u^{(\infty)}\bigr), X_{\pm 2, 1}\bigl(-b^{(\infty)}\bigr)\bigr] &= X_1\bigl(u^{(\infty)} \cdot b^{(\infty)}\bigr),\\ 
\bigl[X_{-1, \pm 2}\bigl(a^{(\infty)}\bigr), X_{\pm 2, 1}\bigl(b^{(\infty)}\bigr)\bigr] &= X_1\bigl(\phi(a^{(\infty)} b^{(\infty)})\bigr),
\end{align*}
so we are done by lemmas \ref{steinberg-presentation}, \ref{ring-generation}, and \ref{form-generation}.

Now suppose that \(\Phi / \{\alpha, \beta\}\) is defined. Any generator \(X_\gamma\) of \(\stunit^{(\infty, S)}(R, \Delta; \Phi)\) factors through \(\stunit^{(\infty, S)}(R, \Delta; \Phi / \{\alpha, \beta\})\) unless \(\gamma\) lies in the span of \(\alpha\) and \(\beta\). We again apply lemmas \ref{steinberg-presentation}, \ref{ring-generation}, and \ref{form-generation}. If \(\Phi \cap (\mathbb R \alpha + \mathbb R\beta)\) is of type \(\mathsf A_2\), then without loss of generality \(\alpha = \mathrm e_{n - 1} - \mathrm e_{n - 2}\) and \(\beta = \gamma = \mathrm e_n - \mathrm e_{n - 1}\), so we use the identity
\[\bigl[X_{n - 1, -1}\bigl(a^{(\infty)}\bigr), X_{-1, n}\bigl(b^{(\infty)}\bigr)\bigr] = X_{n - 1, n}\bigl(a^{(\infty)} b^{(\infty)}\bigr)\]
(here we need that the orthogonal hyperbolic family is strong). If \(\Phi \cap (\mathbb R\alpha + \mathbb R \beta)\) is of type \(\mathsf A_1 \times \mathsf A_1\), then \(n \geq 4\) and we may apply the first claim twice.

Suppose that \(\Phi \cap (\mathbb R \alpha + \mathbb R\beta)\) is of type \(\mathsf A_1 \times \mathsf{BC}_1\). Then without loss of generality \(\alpha = \mathrm e_n - \mathrm e_{n - 1}\), \(\beta = \mathrm e_1\), and \(\gamma\) coincide with one of them. If \(\gamma = \alpha\), then we have the identity
\[\bigl[X_{n - 1, \pm 1}\bigl(a^{(\infty)}\bigr), X_{\pm 1, n}\bigl(b^{(\infty)}\bigr)\bigr] = X_{n - 1, n}\bigl(a^{(\infty)} b^{(\infty)}\bigr).\]
If \(\gamma = \beta\), then there are the identities \begin{align*}
X_{\mp 2, 1}\bigl(-\inv{\rho(u^{(\infty)})} b^{(\infty)}\bigr)\, \bigl[X_{\pm 2}(\dotminus u^{(\infty)}\bigr), X_{\pm 2, 1}\bigl(-b^{(\infty)}\bigr)\bigr] &= X_1\bigl(u^{(\infty)} \cdot b^{(\infty)}\bigr),\\
\bigl[X_{-1, \pm 2}\bigl(a^{(\infty)}\bigr), X_{\pm 2, 1}\bigl(b^{(\infty)}\bigr)\bigr] &= X_1\bigl(\phi(a^{(\infty)} b^{(\infty)})\bigr).
\end{align*}

Finally, suppose that \(\Phi \cap (\mathbb R\alpha + \mathbb R\beta)\) is of type \(\mathsf{BC}_2\). Without loss of generality, \(\alpha = \mathrm e_n - \mathrm e_{n - 1}\), \(\beta = \mathrm e_{n - 1}\), and \(\gamma\) coincides with one of them. If \(\gamma = \alpha\), then
\[\bigl[X_{n - 1, \pm 1}\bigl(a^{(\infty)}\bigr), X_{\pm 1, n}\bigl(b^{(\infty)}\bigr)] = X_{n - 1, n}\bigl(a^{(\infty)} b^{(\infty)}\bigr).\]
If \(\gamma = \beta\), then we apply the identities
\begin{align*}
X_{\mp 1, n - 1}\bigl(-\inv{\rho(u^{(\infty)})} b^{(\infty)}\bigr)\, \bigl[X_{\pm 1}\bigl(\dotminus u^{(\infty)}\bigr), X_{\pm 1, n - 1}\bigl(-b^{(\infty)}\bigr)\bigr] &= X_{n - 1}\bigl(u^{(\infty)} \cdot b^{(\infty)}\bigr),\\
\bigl[X_{1 - n, \pm 1}\bigl(a^{(\infty)}\bigr), X_{\pm 1, n - 1}\bigl(b^{(\infty)}\bigr)\bigr] &= X_{n - 1}\bigl(\phi(a^{(\infty)} b^{(\infty)})\bigr).\qedhere
\end{align*}
\end{proof}

The proof that \(F_\alpha\) is an isomophism is divided in to cases: when \(\alpha\) is short and when \(\alpha\) is ultrashort. The case of short \(\alpha\) is harder, this is done in the next section. To handle this case for \(n = 3\) and strong orthogonal hyperbolic family (say, for the Chevalley groups of types \(\mathsf B_3\) and \(\mathsf C_3\)) we need the following technical lemma. Its proof significantly simplifies if the orthogonal hyperbolic family is also free. Also, the short root case for \(n \geq 4\) follows from the ultrashort root case applied twice.

\begin{lemma}\label{associator}
Suppose that \(n \geq 3\) and the orthogonal hyperbolic family is strong. Let \(G^{(\infty)}\) be a pro-group and
\[\{-, =, \equiv\}_{ij} \colon R^{(\infty, S)}_{2i} \times R^{(\infty, S)}_{ij} \times R^{(\infty, S)}_{j3} \to G^{(\infty)}\]
be morphisms of pro-sets for \(i, j \in \{-1, 1\}\) (``associators''). Suppose that these morphisms satisfy
\begin{enumerate}[label = (A\arabic*)]
\item \(\bigl[\{a^{(\infty)}, b^{(\infty)}, c^{(\infty)}\}_{ij}, \{{a'}^{(\infty)}, {b'}^{(\infty)}, {c'}^{(\infty)}\}_{kl}\bigr] = 1\);
\item \(\{-, =, \equiv\}_{ij}\) are triadditive;
\item \(\{a^{(\infty)}, b^{(\infty)} c^{(\infty)}, d^{(\infty)}\}_{ik} = \{a^{(\infty)} b^{(\infty)}, c^{(\infty)}, d^{(\infty)}\}_{jk}\, \{a^{(\infty)}, b^{(\infty)}, c^{(\infty)} d^{(\infty)}\}_{ij}\);
\item \(\{a^{(\infty)}, b^{(\infty)}, c^{(\infty)}\}_{-i, i} = \{a^{(\infty)}, \inv{b^{(\infty)}}, c^{(\infty)}\}_{-i, i}\);
\item \(\{a^{(\infty)}, b^{(\infty)}, \inv{a^{(\infty)}} c^{(\infty)}\}_{-i, i} = 1\);
\item \(\{a^{(\infty)}, b^{(\infty)} c^{(\infty)}, d^{(\infty)}\}_{ij} = \{\inv{b^{(\infty)}}, \inv{a^{(\infty)}} c^{(\infty)}, d^{(\infty)}\}_{-i, j}\);
\item \(\{a^{(\infty)}, b^{(\infty)}, c^{(\infty)} d^{(\infty)}\}_{ij} = \{\inv{c^{(\infty)}}, \inv{b^{(\infty)}}, \inv{a^{(\infty)}} d^{(\infty)}\}_{-j, -i}^{-1}\)
\end{enumerate}
when both sides are defined. Then \(\{a^{(\infty)}, b^{(\infty)}, c^{(\infty)}\}_{ij} = 1\) for all \(i, j \in \{-1, 1\}\).
\end{lemma}
\begin{proof}
Without loss of generality, \(G\) is an ordinary group. We always use (A1) and (A2) without explicit references. First of all, (A5)--(A7) together with lemma \ref{ring-generation} imply that
\begin{align}
\{(ax)^{(s)}, (\inv y b)^{(s)}, c^{(s)}\}_{ij} &= \{(ay)^{(s)}, (\inv x b)^{(s)}, c^{(s)}\}_{-i, j},  \tag{A8}\\
\{a^{(s)}, (bx)^{(s)}, (\inv y c)^{(s)}\}_{ji} &= \{a^{(s)}, (by)^{(s)}, (\inv x c)^{(s)}\}_{j, -i}, \tag{A9}\\
\{(ax)^{(s)}, b^{(s)}, (\inv xc)^{(s)}\}_{i, -i} &= 1 \tag{A10}
\end{align}
for all \(x, y \in R_{2i}\) and sufficiently large \(s \in S\). From (A8), (A9), and lemma \ref{ring-generation} we get
\begin{align*}
\{(axyb)^{(s)}, c^{(s)}, d^{(s)}\}_{ij} &= \{(ayxb)^{(s)}, c^{(s)}, d^{(s)}\}_{ij},\\
\{a^{(s)}, (bxyc)^{(s)}, d^{(s)}\}_{ij} &= \{a^{(s)}, (byxc)^{(s)}, d^{(s)}\}_{ij},\\
\{a^{(s)}, b^{(s)}, (cxyd)^{(s)}\}_{ij} &= \{a^{(s)}, b^{(s)}, (cyxd)^{(s)}\}_{ij}
\end{align*}
for \(x, y \in R_{22}\) or \(x, y \in R_{-2, -2}\) and sufficiently large \(s \in S\).

Let \(\widetilde R_{22}\) be the factor-ring of \(R_{22}\) by the ideal \(I\) generated by all additive commutators \([x, y] = xy - yx\). Then \(\widetilde R_{22}\) is a unital commutative ring and \(\widetilde R_{21} = R_{21} / IR_{21}\) is a finite projective \(\widetilde R_{22}\)-module with the dual \(\widetilde R_{12} = R_{12} / R_{12} I\). Locally in the Zarisky topology on \(\widetilde R_{22}\) the module \(\widetilde R_{21}\) is free. Hence there are elements \(f_t \in R_{22}\), \(g_t \in R_{22}\), \(u_{tp} \in R_{21}\), \(v_{tp} \in R_{12}\) for \(1 \leq t \leq M\), \(1 \leq p \leq N_t\) such that
\begin{align*}
u_{tp} v_{tq} \equiv 0 &\pmod I \text{ for }p \neq q, &
\sum_p v_{tp} u_{tp} \equiv \sum_r x_r f_t y_r &\pmod{R_{12}IR_{21}},\\
u_{tp} v_{tp} \equiv f_t &\pmod I, &
\sum_t f_t g_t \equiv \sum_t g_t f_t \equiv 1 &\pmod I,
\end{align*}
where \(x_r \in R_{12}\) and \(y_r \in R_{21}\) are fixed elements with \(\sum_r x_r y_r = 1\). Note that \(ab \equiv \bigl(\sum_r x_r b y_r\bigr) a \pmod{R_{12} I}\) and \(bc \equiv c \bigl(\sum_r x_r b y_r \bigr) \pmod{I R_{21}}\) for all \(a \in R_{12}\), \(b \in R_{22}\), \(c \in R_{21}\).

We have to prove that \(\{a^{(\infty)}, b^{(\infty)}, c^{(\infty)}\}_{-1, 1} = 1\), then the other cases follow from (A8), (A9), and lemma \ref{ring-generation}. Inserting \(\sum_r x_r y_r\) or its inverse near the commas and further inserting \(f_t\) between \(x_r\) and \(y_r\) (using \(\sum_t f_t g_t \equiv 1\) and that we may ``swap'' \(g_t\) with \(x_r\) or \(y_r\)), we reduce to the case \(\{a^{(\infty)} \inv{v_{tp}}, \inv{u_{tq}} b^{(\infty)} u_{tq'}, v_{tp'} c^{(\infty)}\}_{-1, 1}\) for various \(t, p, p', q, q'\). If \(p \neq q\) and \(p' \neq q'\), then this morphism is trivial by lemma \ref{ring-generation} and (A3). In the case \(p \neq q'\) and \(p \neq q'\) the morphism is also trivial by (A4). The case \(p = p'\) follows form (A10), and the last case \(q = q'\) follows from (A8)--(A10).
\end{proof}

\section{Elimination of roots: the short root case}

\begin{prop}\label{short-presentation}
Suppose that \(n \geq 4\) or \(n \geq 3\) and the orthogonal hyperbolic family is strong. Then for every short root \(\alpha\) the morphism \(F_\alpha \colon \stunit^{(\infty, S)}(R, \Delta; \Phi / \alpha) \to \stunit^{(\infty, S)}(R, \Delta; \Phi)\) is an isomorphism of pro-groups.
\end{prop}
\begin{proof}
Without loss of generality, \(\alpha = \mathrm e_n - \mathrm e_{n - 1}\). We denote the generators of \(\stunit^{(\infty, S)}(R, \Delta; \Phi / \alpha)\) and \(\stunit^{(\infty, S)}(R, \Delta; \Phi)\) by \(X_{ij}\) and \(X_j\), in the first case \(i, j \in \{\pm 1, \ldots, \pm(n - 2), \pm \infty\}\) and in the second case \(i, j \in \{\pm 1, \ldots, \pm n\}\) (where \(\eta_\infty = \eta_{n - 1} \oplus \eta_n\)). We have to construct morphisms \(\widetilde X_{ij} \colon R^{(\infty, S)}_{i, j} \to \stunit^{(\infty, S)}(R, \Delta; \Phi / \alpha)\), \(\widetilde X_j \colon \Delta^{0, (\infty, S)}_j \to \stunit^{(\infty, S)}(R, \Delta; \Phi / \alpha)\) for \(i, j \in \{\pm 1, \ldots, \pm n\}\) and to show that they satisfy (St0)--(St8).
Let
\begin{itemize}
\item \(\widetilde X_{ij}\bigl(a^{(\infty)}\bigr) = X_{ij}\bigl(a^{(\infty)}\bigr)\) and \(\widetilde X_j\bigl(u^{(\infty)}\bigr) = X_j\bigl(u^{(\infty)}\bigr)\) for \(0 < |i|, |j| < n - 1\);
\item \(\widetilde X_{i, \pm j}\bigl(a^{(\infty)}\bigr) = \widetilde X_{\mp j, -i}\bigl(-\inv{a^{(\infty)}}\bigr) = X_{i, \pm\infty}\bigl(a^{(\infty)}\bigr)\) for \(0 < |i| < n - 1\), \(j \in \{n - 1, n\}\);
\item \(\widetilde X_{\pm j}\bigl(u^{(\infty)}\bigr) = X_{\pm \infty}\bigl(u^{(\infty)}\bigr)\) for \(j \in \{n - 1, n\}\);
\item \(\widetilde X_{\mp(n - 1), \pm n}\bigl(a^{(\infty)}\bigr) = \widetilde X_{\mp n, \pm(n - 1)}\bigl(-\inv{a^{(\infty)}}\bigr) = X_{\pm \infty}\bigl(\phi(a^{(\infty)})\bigr)\);
\item \(\widetilde X^i_{n - 1, n}\bigl(a^{(\infty)}, b^{(\infty)}\bigr) = \bigl[\widetilde X_{n - 1, i}\bigl(a^{(\infty)}\bigr), \widetilde X_{i n}\bigl(b^{(\infty)}\bigr)\bigr]\) for \(0 < |i| < n - 1\), \(a^{(\infty)} \in R^{(\infty, S)}_{n - 1, i}\), \(b^{(\infty)} \in R^{(\infty, S)}_{in}\);
\item \(\widetilde X^\pi_{n - 1, n}\bigl(u^{(\infty)}, v^{(\infty)}\bigr) = \bigl[\widetilde X_{1 - n}\bigl(\dotminus u^{(\infty)}\bigr), \widetilde X_n\bigl(v^{(\infty)}\bigr)\bigr]\) for \(u^{(\infty)} \in \Delta^{0, (\infty, S)}_{1 - n}\), \(v^{(\infty)} \in \Delta^{0, (\infty, S)}_n\);
\item \(\widetilde X^{1 - n}_{n - 1, n}\bigl(u^{(\infty)}, b^{(\infty)}\bigr) = \bigl[\widetilde X_{1 - n}\bigl(u^{(\infty)}\bigr), \widetilde X_{1 - n, n}\bigl(b^{(\infty)}\bigr)\bigr]\, \widetilde X_n\bigl(u^{(\infty)} \cdot (-b^{(\infty)})\bigr)\) for \(u^{(\infty)} \in \Delta^{0, (\infty, S)}_{1 - n}\), \(b^{(\infty)} \in R^{(\infty, S)}_{1 - n, n}\).
\end{itemize}
Now we have \(\widetilde X_\beta\) for all \(\beta \in \Phi \setminus \mathbb R \alpha\), and these morphisms satisfy the Steinberg relations not involving \(\pm \alpha\). Since no Steinberg relation involves \(\alpha\) and \(-\alpha\) simultaneously, it suffices to construct \(\widetilde X_{n - 1, n}\) and prove the Steinberg relations for it. The idea is to find a morphism \(\widetilde X_{n - 1, n}\) such that \(\widetilde X^i_{n - 1, n}\bigl(a^{(\infty)}, b^{(\infty)}\bigr) = \widetilde X_{n - 1, n}\bigl(a^{(\infty)}, b^{(\infty)}\bigr)\), \(\widetilde X^\pi_{n - 1, n}\bigl(u^{(\infty)}, v^{(\infty)}\bigr) = \widetilde X_{n - 1, n}\bigl(\inv{\pi(u^{(\infty)})} \pi(v^{(\infty)})\bigr)\), \(\widetilde X^{1 - n}_{n - 1, n}\bigl(u^{(\infty)}, b^{(\infty)}\bigr) = \widetilde X_{n - 1, n}\bigl(\rho(u^{(\infty)}) b^{(\infty)}\bigr)\) using lemma \ref{ring-presentation}. Note that there is an element \(\sigma\) of the Weyl group with the properties \(\sigma^2 = 1\) and \(\up \sigma \eta_{n - 1} = -\eta_n\), it stabilizes \(\alpha\).

\begin{lemma*}
The Steinberg relations (St3) and (St7) hold for \(\widetilde X^i_{n - 1, n}\), \(\widetilde X^\pi_{n - 1, n}\), \(\widetilde X^{1 - n}_{n - 1, n}\), and \(\widetilde X_\beta\) for \(\beta \in \Phi \setminus \mathbb R\alpha\).
\end{lemma*}
\begin{proof}
We begin with (St3) and (St7). Obviously, \(\widetilde X^i_{n - 1, n}\bigl(a^{(\infty)}, b^{(\infty)}\bigr)\) commutes with \(\widetilde X_{jk}\bigl(c^{(\infty)}\bigr)\) and \(\widetilde X_k\bigl(u^{(\infty)}\bigr)\) for \(0 < |i| < n - 1\) and \(j, -k \notin \{0, 1 - n, n\}\) if in addition \(j, k \neq \pm i\). Also, \(\widetilde X^{1 - n}_{n - 1, n}\bigl(u^{(\infty)}, b^{(\infty)}\bigr)\) and \(\widetilde X^\pi_{n - 1, n}\bigl(u^{(\infty)}, v^{(\infty)}\bigr)\) commute with \(\widetilde X_{ij}\bigl(c^{(\infty)}\bigr)\) for \(0 < |i|, |j| < n - 1\) and \(i \neq \pm j\).

For \(j \neq 0, \pm i, \pm (n - 1), \pm n\) it is easy to see that \(\widetilde X^i_{n - 1, n}\bigl(a^{(\infty)}, b^{(\infty)}\bigr)\) commutes with \(\widetilde X_{j, \pm i}(c^{(\infty)})\), \(\widetilde X_{\pm i}(u^{(\infty)})\), \(\widetilde X_{-i, n}(c^{(\infty)})\), and \(\widetilde X_{n - 1, -i}(c^{(\infty)})\) applying the group-theoretic identities and the known Steinberg relations. A similar computation shows that \(\widetilde X^\pi_{n - 1, n}\bigl(u^{(\infty)}, v^{(\infty)}\bigr)\) and \(\widetilde X^{1 - n}_{n - 1, n}\bigl(u^{(\infty)}, b^{(\infty)}\bigr)\) commute with \(\widetilde X_{n - 1, i}(c^{(\infty)})\), \(\widetilde X_{i n}(c^{(\infty)})\), and \(\widetilde X_i(w^{(\infty)})\) for \(0 < |i| < n - 1\).

Clearly, both \(\widetilde X^\pi\bigl(u^{(\infty)}, v^{(\infty)}\bigr)\) and \(\widetilde X^{1 - n}\bigl(u^{(\infty)}, b^{(\infty)}\bigr)\) commute with
\begin{align*}
\widetilde X_{-i, j}\bigl(\inv{\rho(w^{(\infty)})} d^{(\infty)}\bigr)\, \bigl[\widetilde X_i\bigl(\dotminus w^{(\infty)}\bigr), \widetilde X_{ij}\bigl(-d^{(\infty)}\bigr)\bigr] &= \widetilde X_j\bigl(w^{(\infty)} \cdot d^{(\infty)}\bigr),\\
\bigl[\widetilde X_{-j, i}\bigl(c^{(\infty)}\bigr), \widetilde X_{ij}\bigl(d^{(\infty)}\bigr)\bigr] &= \widetilde X_j\bigl(\phi(c^{(\infty)} d^{(\infty)})\bigr)
\end{align*}
for \(0 < |i| < n - 1\) and \(j \in \{1 - n, n\}\). Hence by lemma \ref{form-generation}, they commute with \(\widetilde X_n\bigl(w^{(\infty)}\bigr)\) and \(\widetilde X_{1 - n}\bigl(w^{(\infty)}\bigr)\).

To finish (St3), we need one simple case of (St5) for \(\widetilde X^i_{n - 1, n}\):
\begin{align*}
\bigl[\widetilde X_{-n, n - 1}(c^{(\infty)}), \widetilde X^i_{n - 1, n}\bigl(a^{(\infty)}, b^{(\infty)}\bigr)\bigr] &= \widetilde X_n\bigl(\phi(c^{(\infty)} a^{(\infty)} b^{(\infty)})\bigr),\\
\bigl[\widetilde X^i_{n - 1, n}\bigl(a^{(\infty)}, b^{(\infty)}\bigr), \widetilde X_{n, 1 - n}(c^{(\infty)})\bigr] &= \widetilde X_{1 - n}\bigl(\phi(a^{(\infty)} b^{(\infty)} c^{(\infty)})\bigr).
\end{align*}

If \(n \geq 4\), then \(\widetilde X^i_{n - 1, n}\bigl(a^{(\infty)}, b^{(\infty)}\bigr)\) commutes with
\begin{align*}
\bigl[\widetilde X_{i, \pm j}\bigl(c^{(\infty)}\bigr), \widetilde X_{\pm j, n}\bigl(d^{(\infty)}\bigr)\bigr] &= \widetilde X_{in}\bigl(c^{(\infty)} d^{(\infty)}\bigr),\\
\bigl[\widetilde X_{n - 1, \pm j}\bigl(c^{(\infty)}\bigr), \widetilde X_{\pm j, i}\bigl(d^{(\infty)}\bigr)\bigr] &= \widetilde X_{n - 1, i}\bigl(c^{(\infty)} d^{(\infty)}\bigr)
\end{align*}
for \(j \neq 0, \pm i, \pm (n - 1), \pm n\). If \(n = 3\) and the orthogonal hyperbolic family is strong, then it commutes with
\begin{align*}
\bigl[\widetilde X_{i, 1 - n}\bigl(c^{(\infty)}\bigr), \widetilde X_{1 - n, n}\bigl(d^{(\infty)}\bigr)\bigr] &= \widetilde X_{in}\bigl(c^{(\infty)} d^{(\infty)}\bigr),\\
\bigl[\widetilde X_{n - 1, -n}\bigl(c^{(\infty)}\bigr), \widetilde X_{-n, i}\bigl(d^{(\infty)}\bigr)\bigr] &= \widetilde X_{n - 1, i}\bigl(c^{(\infty)} d^{(\infty)}\bigr)
\end{align*}
(here we need the proved instance of (St5)). Hence in any case \(\widetilde X^i_{n - 1, n}\bigl(a^{(\infty)}, b^{(\infty)}\bigr)\) commutes with \(\widetilde X_{in}\bigl(c^{(\infty)}\bigr)\) and \(\widetilde X_{n - 1, i}\bigl(c^{(\infty)}\bigr)\) by lemma \ref{ring-generation}. It follows that \(\widetilde X^i_{n - 1, n}\), \(\widetilde X^\pi_{n - 1, n}\), \(\widetilde X^{1 - n}_{n - 1, n}\) commute with each other.
\end{proof}

\begin{lemma*}
The Steinberg relations (St4) and (St5) hold for \(\widetilde X^i_{n - 1, n}\), \(\widetilde X^\pi_{n - 1, n}\), \(\widetilde X^{1 - n}_{n - 1, n}\), and \(\widetilde X_\beta\) for \(\beta \in \Phi \setminus \mathbb R\alpha\).
\end{lemma*}
\begin{proof}
Recall that (St5) is proved for \(\widetilde X^i_{n - 1, n}\) in the previous lemma. Direct calculations show that
\begin{align*}
\bigl[\widetilde X^i_{n - 1, n}\bigl(a^{(\infty)}, b^{(\infty)}\bigr), \widetilde X_{nj}(c^{(\infty)})\bigr] &= \widetilde X_{n - 1, j}\bigl(a^{(\infty)} b^{(\infty)} c^{(\infty)}\bigr),\\
\bigl[\widetilde X_{j, n - 1}(c^{(\infty)}), \widetilde X^i_{n - 1, n}\bigl(a^{(\infty)}, b^{(\infty)}\bigr)\bigr] &= \widetilde X_{jn}\bigl(c^{(\infty)} a^{(\infty)} b^{(\infty)}\bigr)
\end{align*}
for \(j \notin \{0, i, 1 - n, n - 1, -n, n\}\). It follows that
\begin{align*}
\bigl[\widetilde X^i_{n - 1, n}\bigl(a^{(\infty)}, b^{(\infty)}\bigr), \bigl[\widetilde X_{nj}\bigl(c^{(\infty)}\bigr), \widetilde X_{ji}\bigl(d^{(\infty)}\bigr)\bigr]\bigr] &= \widetilde X_{n - 1, i}\bigl(a^{(\infty)} b^{(\infty)} c^{(\infty)} d^{(\infty)}\bigr),\\
\bigl[\bigl[\widetilde X_{ik}\bigl(c^{(\infty)}\bigr), \widetilde X_{k, n - 1}\bigl(d^{(\infty)}\bigr)\bigr], \widetilde X^i_{n - 1, n}\bigl(a^{(\infty)}, b^{(\infty)}\bigr)\bigr] &= \widetilde X_{in}\bigl(c^{(\infty)} d^{(\infty)} a^{(\infty)} b^{(\infty)}\bigr)
\end{align*}
for \(j \notin \{0, \pm i, n - 1, \pm n\}\) and \(k \notin \{0, \pm i, \pm (n - 1), n\}\). Hence we have (St4) for \(\widetilde X^i_{n - 1, n}\) by lemma \ref{ring-generation}.

For \(\widetilde X^\pi_{n - 1, n}\) and \(\widetilde X^{1 - n}_{n - 1, n}\) the proofs of (St4) and (St5) are almost the same. In the case of \(\widetilde X^{1 - n}_{n - 1, n}\) they go as follows. For \(0 < |i| < n - 1\) we have (St4)
\begin{align*}
\bigl[\widetilde X^{1 - n}_{n - 1, n}\bigl(u^{(\infty)}, b^{(\infty)}\bigr), \widetilde X_{ni}(c^{(\infty)})\bigr] &= \widetilde X_{n - 1, i}\bigl(\rho(u^{(\infty)}) b^{(\infty)} c^{(\infty)}\bigr),\\
\bigl[\widetilde X_{i, n - 1}(c^{(\infty)}), \widetilde X^{1 - n}_{n - 1, n}\bigl(u^{(\infty)}, b^{(\infty)}\bigr)\bigr] &= \widetilde X_{in}\bigl(c^{(\infty)} \rho(u^{(\infty)}) b^{(\infty)}\bigr)
\end{align*}
applying the group-theoretic identities and the previous lemma.

Using the proved relations, it is easy to show that
\begin{align*}
\bigl[\widetilde X^{1 - n}_{n - 1, n}\bigl(u^{(\infty)}, b^{(\infty)}\bigr), \bigl[\widetilde X_{ni}
\bigl(c^{(\infty)}\bigr), \widetilde X_{i, 1 - n}\bigl(d^{(\infty)}\bigr)\bigr]\bigr] &= \widetilde X_{1 - n}\bigl(\phi(\rho(u^{(\infty)}) b^{(\infty)} c^{(\infty)} d^{(\infty)})\bigr),\\
\bigl[\bigl[\widetilde X_{-n, i}\bigl(c^{(\infty)}\bigr), \widetilde X_{i, n - 1}\bigl(d^{(\infty)}\bigr)\bigr], \widetilde X^{1 - n}_{n - 1, n}\bigl(u^{(\infty)}, b^{(\infty)}\bigr)\bigr] &= \widetilde X_n\bigl(\phi(c^{(\infty)} d^{(\infty)} \rho(u^{(\infty)}) b^{(\infty)})\bigr).
\end{align*}
for \(0 < |i| < n - 1\). Hence we have (St5) by lemma \ref{ring-generation}.
\end{proof}

\begin{lemma*}
There is a unique morphism \(\widetilde X_{n - 1, n} \colon R_{n - 1, n}^{(\infty, S)} \to \stunit^{(\infty, S)}(R, \Delta; \Phi / \alpha)\) of pro-groups restricting to \(\widetilde X^i_{n - 1, n}\), \(\widetilde X^\pi_{n - 1, n}\), and \(\widetilde X^{1 - n}_{n - 1, n}\).
\end{lemma*}
\begin{proof}
Using the known Steinberg relations, it follows that \(\widetilde X^i_{n - 1, n}\) and \(\widetilde X^\pi_{n - 1, n}\) are biadditive. For \(\widetilde X^{1 - n}_{n - 1, n}\) a similar property is
\begin{align*}
\widetilde X^{1 - n}_{n - 1, n}\bigl(u^{(\infty)}, b^{(\infty)} + {b'}^{(\infty)}\bigr) &= \widetilde X^{1 - n}_{n - 1, n}\bigl(u^{(\infty)}, b^{(\infty)}\bigr)\, \widetilde X^{1 - n}_{n - 1, n}\bigl(u^{(\infty)}, {b'}^{(\infty)}\bigr),\\
\widetilde X^{1 - n}_{n - 1, n}\bigl(u^{(\infty)} \dotplus {u'}^{(\infty)}, b^{(\infty)}\bigr) &= \widetilde X^{1 - n}_{n - 1, n}\bigl({u'}^{(\infty)}, b^{(\infty)}\bigr)\, \widetilde X^\pi_{n - 1, n}\bigl(\dotminus u^{(\infty)}, {u'}^{(\infty)} \cdot b^{(\infty)}\bigr)\, \widetilde X^{1 - n}_{n - 1, n}\bigl(u^{(\infty)}, b^{(\infty)}\bigr).
\end{align*}
Biadditivity implies that
\[\bigl[\widetilde X_{-n, i}\bigl(a^{(\infty)}\bigr), \widetilde X_{i, 1 - n}\bigl(b^{(\infty)}\bigr)\bigr] = \widetilde X^{-i}_{n - 1, n}\bigl(-\inv{b^{(\infty)}}, \inv{a^{(\infty)}}\bigr),\]
so we may apply the element \(\sigma\) from the Weyl group to every identity involving \(X^i_{n - 1, n}\) and \(X^\pi_{n - 1, n}\).

Now we show that the morphisms \(\widetilde X^i_{n - 1, n}\), \(\widetilde X^\pi_{n - 1, n}\), \(\widetilde X^{1 - n}_{n - 1, n}\) are balanced. It is easy to see that
\[\widetilde X^i_{n - 1, n}\bigl(a^{(\infty)}, b^{(\infty)} c^{(\infty)}\bigr) = \widetilde X^j_{n - 1, n}\bigl(a^{(\infty)} b^{(\infty)}, c^{(\infty)}\bigr)\]
for \(0 < |i|, |j| < n - 1\), \(i \neq \pm j\), \(a^{(\infty)} \in R^{(\infty, S)}_{n - 1, i}\), \(b^{(\infty)} \in R^{(\infty, S)}_{ij}\), \(c^{(\infty)} \in R^{(\infty, S)}_{jn}\);
\[\widetilde X^{1 - n}_{n - 1, n}\bigl(u^{(\infty)}, b^{(\infty)} c^{(\infty)}\bigr) = \widetilde X^i_{n - 1, n}\bigl(\rho(u^{(\infty)}) b^{(\infty)}, c^{(\infty)}\bigr)\]
for \(0 < |i| < n - 1\), \(u^{(\infty)} \in \Delta^{0, (\infty, S)}_{1 - n}\), \(b^{(\infty)} \in R^{(\infty, S)}_{1 - n, i}\), \(c^{(\infty)} \in R^{(\infty, S)}_{in}\);
\[\widetilde X^i_{n - 1, n}\bigl(a^{(\infty)} \rho(u^{(\infty)}), c^{(\infty)}\bigr) = \widetilde X^{-i}_{n - 1, n}\bigl(a^{(\infty)}, \rho(u^{(\infty)}) c^{(\infty)}\bigr)\]
for \(a^{(\infty)} \in R^{(\infty, S)}_{n - 1, -i}\), \(u^{(\infty)} \in \Delta^{0, (\infty, S)}_i\), \(c^{(\infty)} \in R^{(\infty, S)}_{in}\);
\[\widetilde X^\pi_{n - 1, n}\bigl(u^{(\infty)}, v^{(\infty)} \cdot b^{(\infty)}\bigr) = \widetilde X^i_{n - 1, n}\bigl(\inv{\pi(u^{(\infty)})} \pi(v^{(\infty)}), b^{(\infty)}\bigr)\]
for \(0 < |i| < n - 1\), \(u^{(\infty)} \in \Delta^{0, (\infty, S)}_{1 - n}\), \(v^{(\infty)} \in \Delta^{0, (\infty, S)}_i\), \(b \in R^{(\infty, S)}_{in}\);
\[\widetilde X^\pi_{n - 1, n}\bigl(u^{(\infty)}, \phi(a^{(\infty)} b^{(\infty)})\bigr) = 1\]
for \(0 < |i| < n - 1\), \(u^{(\infty)} \in \Delta^{0, (\infty, S)}_{1 - n}\), \(a^{(\infty)} \in R^{(\infty, S)}_{-n, i}\), \(b^{(\infty)} \in R^{(\infty, S)}_{in}\).

Hence by lemma \ref{ring-generation} it follows that \(\widetilde X^\pi_{n - 1, n}\bigl(u^{(\infty)}, \phi(a^{(\infty)})\bigr) = 1\).
Using that \(\widetilde X^\pi_{n - 1, n}\) is balanced (and the symmetry \(\sigma\)), it is easy to prove that
\[\widetilde X^{1 - n}_{n - 1, n}\bigl(u^{(\infty)} \cdot a^{(\infty)}, b^{(\infty)}\bigr) = \widetilde X^i_{n - 1, n}\bigl(\inv{a^{(\infty)}} \rho(u^{(\infty)}), a^{(\infty)} b^{(\infty)}\bigr)\]
for \(0 < |i| < n - 1\), \(u^{(\infty)} \in \Delta^{0, (\infty, S)}_i\), \(a^{(\infty)} \in R^{(\infty, S)}_{i, 1 - n}\), \(b^{(\infty)} \in R^{(\infty, S)}_{1 - n, n}\). Finally,
\[\widetilde X^{1 - n}_{n - 1, n}\bigl(\phi(a^{(\infty)} b^{(\infty)}), c^{(\infty)}\bigr) = \widetilde X^i_{n - 1, n}\bigl(a^{(\infty)}, b^{(\infty)} c^{(\infty)}\bigr)\, \widetilde X^{-i}_{n - 1, n}\bigl(\inv{b^{(\infty)}}, \inv{a^{(\infty)}} c^{(\infty)}\bigr)^{-1}\]
for \(0 < |i| < n - 1\), \(a^{(\infty)} \in R^{(\infty, S)}_{n - 1, i}\), \(b^{(\infty)} \in R^{(\infty, S)}_{i, 1 - n}\), \(c^{(\infty)} \in R^{(\infty, S)}_{1 - n, n}\).

If \(n \geq 4\), then
\[\widetilde X^i_{n - 1, n}\bigl(a^{(\infty)} b^{(\infty)} c^{(\infty)}, d^{(\infty)}\bigr) = \widetilde X^j_{n - 1, n}\bigl(a^{(\infty)}, b^{(\infty)} c^{(\infty)} d^{(\infty)}\bigr)\]
for \(i, j \in \{-1, 1\}\), \(a^{(\infty)} \in R^{(\infty, S)}_{n - 1, j}\), \(b^{(\infty)} \in R^{(\infty, S)}_{j, \pm 2}\), \(c^{(\infty)} \in R^{(\infty, S)}_{\pm 2, i}\), \(d^{(\infty)} \in R^{(\infty, S)}_{in}\). Hence by lemma \ref{ring-presentation} there is a unique morphism \(\widetilde X_{n - 1, n} \colon R^{(\infty, S)}_{n - 1, n} \to \stunit^{(\infty, S)}(R, \Delta; \Phi / \alpha)\) of pro-groups restricting to \(\widetilde X_{n - 1, n}^{\pm 1}\).

Now suppose that \(n = 3\) and the orthogonal hyperbolic family is strong. Let
\[\{a^{(\infty)}, b^{(\infty)}, c^{(\infty)}\}_{ij} = \widetilde X^j_{n - 1, n}\bigl(a^{(\infty)} b^{(\infty)}, c^{(\infty)}\bigr)\, \widetilde X^i_{n - 1, n}\bigl(a^{(\infty)}, b^{(\infty)} c^{(\infty)}\bigr)^{-1}\]
for \(i, j \in \{-1, 1\}\), \(a^{(\infty)} \in R^{(\infty, S)}_{n - 1, i}\), \(b^{(\infty)} \in R^{(\infty, S)}_{ij}\), \(c^{(\infty)} \in R^{(\infty, S)}_{jn}\). This morphisms satisfy all the axioms from lemma \ref{associator}: the first three are trivial, and the rest follows from the balancing properties of \(\widetilde X^i_{n - 1, n}\bigl(a^{(\infty)} \rho(\phi(b^{(\infty)})), c^{(\infty)}\bigr)\), \(\widetilde X^{1 - n}_{n - 1, n}\bigl(\phi(a^{(\infty)}) \cdot b^{(\infty)}, c^{(\infty)}\bigr)\), \(\widetilde X^{1 - n}_{n - 1, n}(\phi(a^{(\infty)} b^{(\infty)}), c^{(\infty)} d^{(\infty)})\), and \(\widetilde X^{1 - n}_{n - 1, n}\bigl(\phi(a^{(\infty)} b^{(\infty)} c^{(\infty)}), d^{(\infty)}\bigr)\). Hence \(\{a^{(\infty)}, b^{(\infty)}, c^{(\infty)}\}_{ij} = 1\) and by lemma \ref{ring-presentation} there is a unique morphism \(\widetilde X_{n - 1, n} \colon R^{(\infty, S)}_{n - 1, n} \to \stunit^{(\infty, S)}(R, \Delta; \Phi / \alpha)\) of pro-groups restricting to \(\widetilde X_{n - 1, n}^{\pm 1}\).

In any case, lemmas \ref{ring-generation} and \ref{form-generation} together with the balancing properties imply that \(\widetilde X_{n - 1, n}\) restricts to all \(\widetilde X^i_{n - 1, n}\), \(\widetilde X^\pi_{n - 1, n}\), and \(\widetilde X^{1 - n}_{n - 1, n}\).
\end{proof}

Hence \(\widetilde X_{n - 1, n}\) exists and satisfies (St0)--(St7) by lemma \ref{ring-generation}. Also (St8) holds by definition if \(\alpha\) appears in the right hand side (there is another such case of (St8) obtained by using the symmetry \(\sigma\)). The remaining case of (St8) follows from a direct calculation
\begin{align*}
\bigl[\widetilde X_{n - 1}(u^{(\infty)}), \widetilde X^i_{n - 1, n}\bigl(a^{(\infty)}, b^{(\infty)}\bigr)\bigr] &= \widetilde X_{1 - n, n}\bigl(\rho(u^{(\infty)}) a^{(\infty)} b^{(\infty)}\bigr)\, \widetilde X_n\bigl(\dotminus u^{(\infty)} \cdot (-a^{(\infty)} b^{(\infty)})\bigr)
\end{align*}
and lemma \ref{ring-generation}. By lemma \ref{steinberg-presentation}, there is a morphism \(G_\alpha \colon \stunit^{(\infty, S)}(R, \Delta; \Phi) \to \stunit^{(\infty, S)}(R, \Delta; \Phi / \alpha)\) of pro-groups such that \(G_\alpha \circ X_{ij} = \widetilde X_{ij}\) and \(G_\alpha \circ X_j = \widetilde X_j\). Clearly, \(G_\alpha \circ F_\alpha = \id\). Hence by lemma \ref{root-generation} also \(F_\alpha \circ G_\alpha = \id\).
\end{proof}

\section{Elimination of roots: the ultrashort root case}

\begin{prop}\label{ultrashort-presentation}
Suppose that \(n \geq 3\). Then for every ultrashort root \(\alpha\) the morphism \(F_\alpha \colon \stunit^{(\infty, S)}(R, \Delta; \Phi / \alpha) \to \stunit^{(\infty, S)}(R, \Delta; \Phi)\) is an isomorphism of pro-groups.
\end{prop}
\begin{proof}
Without loss of generality, \(\alpha = \mathrm e_1\). We denote the generators of \(\stunit^{(\infty, S)}(R, \Delta; \Phi / \alpha)\) and \(\stunit^{(\infty, S)}(R, \Delta; \Phi)\) by \(X_{ij}\) and \(X_j\), in the first case \(i, j \in \{\pm 2, \ldots, \pm n\}\) and in the second case \(i, j \in \{\pm 1, \ldots, \pm n\}\). We have to construct morphisms \(\widetilde X_{ij} \colon R^{(\infty, S)}_{i, j} \to \stunit^{(\infty, S)}(R, \Delta; \Phi / \alpha)\), \(\widetilde X_j \colon \Delta^{0, (\infty, S)}_j \to \stunit^{(\infty, S)}(R, \Delta; \Phi / \alpha)\) for \(i, j \in \{\pm 1, \ldots, \pm n\}\) and to show that they satisfy (St0)--(St8).
Let
\begin{itemize}
\item \(\widetilde X_{ij}\bigl(a^{(\infty)}\bigr) = X_{ij}\bigl(a^{(\infty)}\bigr)\) and \(\widetilde X_j\bigl(u^{(\infty)}\bigr) = X_j\bigl(u^{(\infty)}\bigr)\) for \(i, j \neq 0, \pm 1\);
\item \(\widetilde X_{\pm 1, j}\bigl(a^{(\infty)}\bigr) = \widetilde X_{-j, \mp 1}\bigl(-\inv{a^{(\infty)}}\bigr) = X_j\bigl(q_{\pm 1} \cdot a^{(\infty)}\bigr)\) for \(j \neq 0, \pm 1\);
\item \(\widetilde X^i_{-1, 1}\bigl(a^{(\infty)}, b^{(\infty)}\bigr) = \bigl[\widetilde X_{-1, i}
\bigl(a^{(\infty)}\bigr), \widetilde X_{i1}\bigl(b^{(\infty)}\bigr)\bigr]\) for \(i \neq 0, \pm 1\), \(a^{(\infty)} \in R^{(\infty, S)}_{-1, i}\), \(b^{(\infty)} \in R^{(\infty, S)}_{i1}\);
\item \(\widetilde X^i_1\bigl(u^{(\infty)}, b^{(\infty)}\bigr) = \widetilde X_{-i, 1}\bigl(\inv{\rho(u^{(\infty)})} b^{(\infty)}\bigr)\, \bigl[\widetilde X_i\bigl(\dotminus u^{(\infty)}\bigr), \widetilde X_{i1}\bigl(-b^{(\infty)}\bigr)\bigr]\) for \(i \neq 0, \pm 1\), \(u^{(\infty)} \in \widetilde \Delta^{0, (\infty, S)}_i\), \(b^{(\infty)} \in \widetilde R^{(\infty, S)}_{i1}\).
\end{itemize}
Now we have \(\widetilde X_\beta\) for all \(\beta \in \Phi \setminus \mathbb R \alpha\), and these morphisms satisfy the Steinberg relations not involving \(\mathbb R \alpha\). Since no Steinberg relation involves \(\alpha\) and \(-\alpha\) simultaneously, it suffices to construct \(\widetilde X_1\) and prove the Steinberg relations for it. The idea is to find a morphism \(\widetilde X_1\) such that \(\widetilde X^i_{-1, 1}\bigl(a^{(\infty)}, b^{(\infty)}\bigr) = \widetilde X_1\bigl(\phi(a^{(\infty)} b^{(\infty)})\bigr)\) and \(\widetilde X^i_1\bigl(u^{(\infty)}, b^{(\infty)}\bigr) = \widetilde X_1\bigl(u^{(\infty)} \cdot b^{(\infty)}\bigr)\) using lemma \ref{form-presentation}.

\begin{lemma*}
Morphisms \(\widetilde X^i_{-1, 1}\), \(\widetilde X^i_1\), and \(\widetilde X_\beta\) for \(\beta \in \Phi \setminus \mathbb R\alpha\) satisfy (St6) and (St7).
\end{lemma*}
\begin{proof}
We begin with (St7). Clearly, \(\widetilde X_{-1, 1}^i\bigl(a^{(\infty)}, b^{(\infty)}\bigr)\) and \(\widetilde X_1^i\bigl(u^{(\infty)}, b^{(\infty)}\bigr)\) commute with \(\widetilde X_{jk}\bigl(c^{(\infty)}\bigr)\) if \(j\) and \(k\) are different from \(0, \pm i\) and \(j \neq 1 \neq -k\). It is easy to check that they commute with \(\widetilde X_{\pm i, k}(c^{(\infty)})\) by direct calculations for \(k \neq 0, \pm 1, \pm i\). Since they commute with
\[\bigl[\widetilde X_{\pm i, j}\bigl(c^{(\infty)}\bigr), \widetilde X_{j1}\bigl(d^{(\infty)}\bigr)\bigr] = \widetilde X_{\pm i, 1}\bigl(c^{(\infty)} d^{(\infty)}\bigr)\]
for \(j \neq 0, \pm 1, \pm i\), they also commute with \(\widetilde X_{\pm i, 1}\bigl(c^{(\infty)}\bigr)\) by lemma \ref{ring-generation}. By (St0), they also commute with \(\widetilde X_{j, \pm i}\bigl(c^{(\infty)}\bigr)\) for \(j \notin \{0, 1, -i, i\}\).

Now we prove (St6). If \(j \neq 0, \pm 1, \pm i\), then clearly
\begin{align*}
\bigl[\widetilde X_j(v^{(\infty)}), \widetilde X^i_{-1, 1}\bigl(a^{(\infty)}, b^{(\infty)}\bigr)\bigr] &= 1,\\
\bigl[\widetilde X_j(v^{(\infty)}), \widetilde X^i_1\bigl(u^{(\infty)}, b^{(\infty)}\bigr)\bigr] &= \widetilde X_{-j, 1}\bigl(-\inv{\pi(v^{(\infty)})} \pi(u^{(\infty)}) b^{(\infty)}\bigr).
\end{align*}
Since \(\bigl[\widetilde X_{\mp i, j}\bigl(c^{(\infty)}\bigr), \widetilde X_{j, \pm i}\bigl(d^{(\infty)}\bigr)\bigr] = \widetilde X_{\pm i}\bigl(\phi(c^{(\infty)} d^{(\infty)})\bigr)\) commutes with \(\widetilde X^i_{-1, 1}\bigl(a^{(\infty)}, b^{(\infty)}\bigr)\) and \(\widetilde X^i_1\bigl(u^{(\infty)}, b^{(\infty)}\bigr)\) for \(j \neq 0, \pm 1, \pm i\), it follows by lemma \ref{ring-generation} that \(\widetilde X_{\pm i}\bigl(\phi(c^{(\infty)})\bigr)\) commutes with these morphisms. Since
\[\bigl[\widetilde X_j\bigl(\dotminus v^{(\infty)}\bigr), \widetilde X_{j, \pm i}\bigl(-d^{(\infty)}\bigr)\bigr] = \widetilde X_{-j, \pm i}\bigl(-\inv{\rho(v^{(\infty)})} d^{(\infty)}\bigr)\, \widetilde X_{\pm i}\bigl(v^{(\infty)} \cdot d^{(\infty)}\bigr),\]
we have
\begin{align*}
\bigl[\widetilde X^i_{-1, 1}(a^{(\infty)}, b^{(\infty)}), \widetilde X_{\pm i}\bigl(v^{(\infty)} \cdot d^{(\infty)}\bigr)\bigr] &= \widetilde X_{-1, \pm i}\bigl(-\inv{b^{(\infty)}} \inv{\pi(u^{(\infty)})} \pi(v^{(\infty)}) d^{(\infty)}\bigr)\\
\bigl[\widetilde X^i_1(u^{(\infty)}, b^{(\infty)}), \widetilde X_{\pm i}\bigl(v^{(\infty)} \cdot d^{(\infty)}\bigr)\bigr] &= \widetilde X_{-1, \pm i}\bigl(-\inv{b^{(\infty)}} \inv{\pi(u^{(\infty)})} \pi(v^{(\infty)}) d^{(\infty)}\bigr).
\end{align*}
for \(j \neq 0, \pm 1, \pm i\), \(v^{(\infty)} \in \Delta^{0, (\infty, S)}_j\), \(d^{(\infty)} \in R^{(\infty, S)}_{j, \pm i}\). Lemma \ref{form-generation} now implies that \begin{align*}
\bigl[\widetilde X^i_{-1, 1}\bigl(a^{(\infty)}, b^{(\infty)}\bigr), \widetilde X_{\pm i}\bigl(v^{(\infty)}\bigr)\bigr] &= 1,\\
\bigl[\widetilde X^i_1\bigl(u^{(\infty)}, b^{(\infty)}\bigr), X_{\pm i}\bigl(v^{(\infty)}\bigr)\bigr] &= \widetilde X_{-1, \pm i}\bigl(-\inv{b^{(\infty)}} \inv{\pi(u^{(\infty)})} \pi(v^{(\infty)})\bigr).\qedhere
\end{align*}
\end{proof}

\begin{lemma*}
There exists a unique morphism \(\widetilde X_{-1, 1} \colon R^{(\infty, S)}_{-1, 1} \to \stunit^{(\infty, S)}(R, \Delta; \Phi / \alpha)\) of pro-groups such that \(\widetilde X^i_{-1, 1}\bigl(a^{(\infty)}, b^{(\infty)}\bigr) = \widetilde X_{-1, 1}\bigl(a^{(\infty)} b^{(\infty)}\bigr)\). It commutes with \(\widetilde X^j_1\) and satisfies \(\widetilde X_{-1, 1}\bigl(a^{(\infty)}\bigr) = \widetilde X_{-1, 1}\bigl(-\inv{a^{(\infty)}}\bigr)\).
\end{lemma*}
\begin{proof}
Note that \(\widetilde X^i_{-1, 1}\bigl(a^{(\infty)}, b^{(\infty)}\bigr)\) commutes with \(\widetilde X^j_{-1, 1}\bigl(c^{(\infty)}, d^{(\infty)}\bigr)\) and \(\widetilde X^j_1\bigl(u^{(\infty)}, d^{(\infty)}\bigr)\) for all possible \(i, j\). By the previous lemma,
\begin{align*}
\widetilde X^i_{-1, 1}\bigl(a^{(\infty)}, b^{(\infty)} + {b'}^{(\infty)}\bigr) &= \widetilde X^i_{-1, 1}\bigl(a^{(\infty)}, b^{(\infty)}\bigr)\, \widetilde X_{i1}\bigl(a^{(\infty)}, {b'}^{(\infty)}\bigr),\\
\widetilde X^i_{-1, 1}\bigl(a^{(\infty)}, b^{(\infty)}\bigr) &= \widetilde X^{-i}_{-1, 1}\bigl(-\inv{b^{(\infty)}}, \inv{a^{(\infty)}}\bigr).
\end{align*}
If \(\pm 1, \pm i, \pm j\) are all distinct and non-zero, the it is easy to see that
\[\widetilde X^i_{-1, 1}\bigl(a^{(\infty)}, b^{(\infty)} c^{(\infty)}\bigr) = \widetilde X^j_{-1, 1}\bigl(a^{(\infty)} b^{(\infty)}, c^{(\infty)}\bigr)\]
for \(a^{(\infty)} \in R^{(\infty, S)}_{-1, i}\), \(b^{(\infty)} \in R^{(\infty, S)}_{ij}\), \(c^{(\infty)} \in R^{(\infty, S)}_{j1}\). Let \(i \neq 0, \pm 1, \pm 2\). Since
\[\widetilde X^l_{-1, 1}\bigl(a^{(\infty)}, b^{(\infty)} c^{(\infty)} d^{(\infty)}\bigr) = \widetilde X^{l'}_{-1, 1}\bigl(a^{(\infty)} b^{(\infty)} c^{(\infty)}, d^{(\infty)}\bigr)\]
for \(l, l' \in \{-2, 2\}\), \(a^{(\infty)} \in R^{(\infty, S)}_{-1, l}\), \(b^{(\infty)} \in R^{(\infty, S)}_{l, \pm i}\), \(c^{(\infty)} \in R^{(\infty, S)}_{\pm i, l'}\), \(d^{(\infty)} \in R^{(\infty, S)}_{l1}\), it follows from lemma \ref{ring-generation} that
\[\widetilde X^l_{-1, 1}\bigl(a^{(\infty)}, b^{(\infty)} c^{(\infty)}\bigr) = \widetilde X^{l'}_{-1, 1}\bigl(a^{(\infty)} b^{(\infty)}, c^{(\infty)}\bigr)\]
for \(l, l' \in \{-2, 2\}\), \(a^{(\infty)} \in R^{(\infty, S)}_{-1, l}\), \(b^{(\infty)} \in R^{(\infty, S)}_{ll'}\), \(c^{(\infty)} \in R^{(\infty, S)}_{l'1}\). By the same lemma, there is a morphism \(\widetilde X_{-1, 1} \colon R^{(\infty, S)}_{-1, 1} \to \stunit^{(\infty, S)}(R, \Delta; \Phi / \alpha)\) of pro-groups such that \(\widetilde X^{\pm 2}_{-1, 1}\bigl(a^{(\infty)}, b^{(\infty)}\bigr) = \widetilde X_{-1, 1}\bigl(a^{(\infty)} b^{(\infty)}\bigr)\). It satisfies all claims from the statement by lemma \ref{ring-generation}.
\end{proof}

\begin{lemma*}
There exists a unique morphism \(\widetilde X_1 \colon \Delta^{0, (\infty, S)}_i \to \stunit^{(\infty, S)}(R, \Delta; \Phi / \alpha)\) of pro-groups restricting to \(\widetilde X_{-1, 1}\) and \(\widetilde X^i_1\).
\end{lemma*}
\begin{proof}
There is a commutation relation
\begin{align*}
\bigl[\widetilde X^i_1(u^{(\infty)}, b^{(\infty)}), \widetilde X^j_1\bigl(v^{(\infty)}, c^{(\infty)}\bigr)\bigr] &= \widetilde X_{-1, 1}\bigl(-\inv{b^{(\infty)}} \inv{\pi(u^{(\infty)})} \pi(v^{(\infty)}) c^{(\infty)}\bigr)
\end{align*}
for all \(i, j \notin \{0, -1, 1\}\). Also,
\begin{align*}
\widetilde X^i_1\bigl(u^{(\infty)} \dotplus {u'}^{(\infty)}, b^{(\infty)}\bigr) &= \widetilde X_1^i\bigl(u^{(\infty)}, b^{(\infty)}\bigr)\, \widetilde X_1^i\bigl({u'}^{(\infty)}, b^{(\infty)}\bigr),\\
\widetilde X^i_1\bigl(u^{(\infty)}, b^{(\infty)} + {b'}^{(\infty)}\bigr) &= \widetilde X^i_1\bigl(u^{(\infty)}, b^{(\infty)}\bigr)\, \widetilde X_{-1, 1}\bigl(\inv{{b'}^{(\infty)}} \rho(u^{(\infty)}) b^{(\infty)}\bigr)\, \widetilde X^i_1\bigl(u^{(\infty)}, {b'}^{(\infty)}\bigr).
\end{align*}
If \(\pm 1, \pm i, \pm j\) are all distinct and non-zero, then it is easy to check that
\begin{align*}
\widetilde X^i_1\bigl(u^{(\infty)}, b^{(\infty)} c^{(\infty)}\bigr) &= \widetilde X^j_1\bigl(u^{(\infty)} \cdot b^{(\infty)}, c^{(\infty)}\bigr),\\
\widetilde X^j_1\bigl(\phi(a^{(\infty)} b^{(\infty)}), c^{(\infty)}\bigr) &= \widetilde X_{-1, 1}\bigl(\inv{c^{(\infty)}} a^{(\infty)} b^{(\infty)} c^{(\infty)}\bigr)
\end{align*}
for \(u^{(\infty)} \in \Delta^{0, (\infty, S)}_i\), \(a^{(\infty)} \in R^{(\infty, S)}_{-j, i}\), \(b^{(\infty)} \in R^{(\infty, S)}_{ij}\), \(c^{(\infty)} \in R^{(\infty, S)}_{j1}\). Let \(i \neq 0, \pm 1, \pm 2\). Since
\[\widetilde X^l_1\bigl(u^{(\infty)}, b^{(\infty)} c^{(\infty)} d^{(\infty)}\bigr) = \widetilde X^{l'}_1\bigl(u^{(\infty)} \cdot b^{(\infty)} c^{(\infty)}, d^{(\infty)}\bigr)\]
for \(l, l' \in \{-2, 2\}\), \(u^{(\infty)} \in \Delta^{0, (\infty, S)}_l\), \(b^{(\infty)} \in R^{(\infty, S)}_{l, \pm i}\), \(c^{(\infty)} \in R^{(\infty, S)}_{\pm i, l'}\), \(d^{(\infty)} \in R^{(\infty, S)}_{l'1}\), it follows from lemma \ref{ring-generation} that
\[\widetilde X^l\bigl(u^{(\infty)}, b^{(\infty)} c^{(\infty)}\bigr) = \widetilde X^{l'}\bigl(u^{(\infty)} \cdot b^{(\infty)}, c^{(\infty)}\bigr)\]
for \(l, l' \in \{-2, 2\}\), \(u^{(\infty)} \in \Delta^{0, (\infty, S)}_l\), \(b^{(\infty)} \in R^{(\infty, S)}_{ll'}\), \(c^{(\infty)} \in R^{(\infty, S)}_{l'1}\). Similarly,
\[\widetilde X^{\pm 2}_1\bigl(\phi(a^{(\infty)}), b^{(\infty)}\bigr) = \widetilde X_{-1, 1}\bigl(\inv{b^{(\infty)}} a^{(\infty)} b^{(\infty)}\bigr).\]
Hence by lemma \ref{form-presentation} there is a morphism \(\widetilde X_1 \colon \Delta^{0, (\infty, S)}_1 \to \stunit^{(\infty, S)}(R, \Delta; \Phi / \alpha)\) of pro-groups such that \(\widetilde X^{\pm 2}_1\bigl(u^{(\infty)}, b^{(\infty)}\bigr) = \widetilde X_1\bigl(u^{(\infty)} \cdot b^{(\infty)}\bigr)\) and \(\widetilde X_{-1, 1}\bigl(a^{(\infty)}\bigr) = \widetilde X_1\bigl(\phi(a^{(\infty)})\bigr)\). Moreover, \(\widetilde X^i_1\bigl(u^{(\infty)}, b^{(\infty)}\bigr) = \widetilde X_1\bigl(u^{(\infty)} \cdot b^{(\infty)}\bigr)\) for all \(i \neq 0, \pm 1\).
\end{proof}

Clearly, \(\widetilde X_1\) satisfies (St2), (St5), (St6), (St7), and the case of (St8) where \(\widetilde X_1\) appears in the right hand side. For any \(i \neq 0, \pm 1\) choose some \(j \neq 0, \pm 1, \pm i\). The remaining case of (St8) follows from the calculations
\begin{align*}
\bigl[\widetilde X^j_{-1, 1}\bigl(a^{(\infty)}, b^{(\infty)}\bigr), \widetilde X_{1i}(c^{(\infty)})\bigr] &= \widetilde X_{-1, i}\bigl((a^{(\infty)} b^{(\infty)} - \inv{a^{(\infty)} b^{(\infty)}}) c^{(\infty)}\bigr)\, \widetilde X_i\bigl(\dotminus \phi(a^{(\infty)} b^{(\infty)}) \cdot c^{(\infty)}\bigr),\\
\bigl[\widetilde X^j_1\bigl(u^{(\infty)}, b^{(\infty)}\bigr), \widetilde X_{1i}(c^{(\infty)})\bigr] &= \widetilde X_{-1, i}\bigl(\rho(u^{(\infty)} \cdot b^{(\infty)}) c^{(\infty)}\bigr)\, \widetilde X_i\bigl(\dotminus u^{(\infty)} \cdot (-b^{(\infty)} c^{(\infty)})\bigr)
\end{align*}
and lemma \ref{form-generation}.

From lemma \ref{steinberg-presentation} if follows that there is a unique morphism \(G_\alpha \colon \stunit(R, \Delta; \Phi)^{(\infty)} \to \stunit(R, \Delta; \Phi / \alpha)^{(\infty)}\) of pro-groups such that \(G_\alpha \circ X_{ij} = \widetilde X_{ij}\) and \(G_\alpha \circ X_j = \widetilde X_j\). Moreover, \(G_\alpha \circ F_\alpha = \id\). Since \(F_\alpha\) is an epimorphism of pro-groups, it is an isomorphism with the inverse \(G_\alpha\).
\end{proof}

\section{Actions of local groups}

By lemma \ref{ofa-homotopes}, the group \(\diag(S^{-1} R, S^{-1} \Delta; \Phi)\) acts on \(\stunit^{(\infty, S)}(R, \Delta; \Phi)\) and on \((R^{(\infty, S)}, \Delta^{(\infty, S)})\) by automorphisms making \(\stmap\) equivariant. By propositions \ref{short-presentation} and \ref{ultrashort-presentation}, for each \(\alpha \in \Phi\) the group \(\diag(S^{-1} R, S^{-1} \Delta; \Phi / \alpha)\) also acts on \(\stunit^{(\infty, S)}(R, \Delta; \Phi)\) making \(\stmap\) equivariant under an assumption on \(n\). Moreover, under an assumption on \(n\) every element from \(\diag(S^{-1} R, S^{-1} \Delta; \Phi / \alpha) * \diag(S^{-1} R, S^{-1} \Delta; \Phi / \beta)\) with trivial image in \(\unit(S^{-1} R, S^{-1} \Delta)\) acts trivially on the Steinberg pro-group by lemma \ref{root-generation} if \(\Phi / \{\alpha, \beta\}\) is defined. In particular, \(\stunit(S^{-1} R, S^{-1} \Delta) \rtimes \diag(S^{-1} R, S^{-1} \Delta)\) acts on \(\stunit^{(\infty, S)}(R, \Delta)\) making \(\stmap\) equivariant if \(n \geq 4\) or \(n \geq 3\) and the orthogonal hyperbolic family is strong. Such actions are compatible for various \(S\): if \(S \subseteq S'\) are two multiplicative subsets, then the morphism \(\stunit^{(\infty, S')}(R, \Delta; \Phi) \to \stunit^{(\infty, S)}(R, \Delta; \Phi)\) of pro-groups is equivariant under the action of each \(\diag(S^{-1} R, S^{-1} \Delta; \Phi / \alpha)\) under an assumption on \(n\).

Recall that a sequence \(a_1, \ldots, a_k\) in a unital ring \(A\) is called left unimodular if there is a sequence \(b_1, \ldots, b_k\) in \(A\) such that \(\sum_i b_i a_i = 1\). The ring \(A\) satisfies \(\mathrm{sr}(A) \leq k - 1\) if for every left unimodular sequence \(a_1, \ldots, a_k\) there are elements \(c_1, \ldots, c_{k - 1} \in A\) such that \(a_1 + c_1 a_k, \ldots, a_{k - 1} + c_{k - 1} a_k\) is also unimodular. More generally, let \(M_i\) be right \(A\)-modules for \(1 \leq i \leq k\). A sequence \(m_1, \ldots, m_k\) for \(m_i \in M_i\) is called left unimodular if there are homomorphisms \(f_i \colon M_i \to A\) such that \(\sum_i f_i(m_i) = 1\).

Now let \((R, \Delta)\) be an odd form \(K\)-algebra with a free orthogonal hyperbolic family \(\eta_1, \ldots, \eta_n\). Let \(\Lambda = \{\rho(u) \mid u \in \Delta^0_{-1}, \pi(u) = 0\}\), this is an even form parameter in the sense that \(\{a - \inv a \mid a \in R_{1, -1}\} \leq \Lambda \leq \{a \in R_{1, -1} \mid a + \inv a = 0\}\) and \(\inv a \Lambda a \leq \Lambda\) for any \(a \in R_{11}\). We say that \(\Lambda\mathrm{sr}(\eta_1; R, \Delta) \leq k - 1\) if \(\mathrm{sr}(R_{11}) \leq k - 1\) and for every unimodular sequence \(a_{-k}, \ldots, a_{-1}, b_1, \ldots, b_k\) with \(a_k \in R_{-1, 1}\) and \(b_k \in R_{11}\) there is a matrix \(\{c_{ij} \in R_{1, -1}\}_{i, -j = 1}^k\) such that \(c_{ij} = -\inv{c_{-j, -i}}\), \(c_{i, -i} \in \Lambda\), and the sequence \(a_1 + \sum_i c_{1i} b_i, \ldots, a_k + \sum_i c_{ki} b_i\) is left unimodular in \(R_{11}\). For example, if there are elements \(e_{1, -1} \in R_{1, -1}\) and \(e_{-1, 1} \in R_{-1, 1}\) such that \(e_{1, -1} e_{-1, 1} = e_1\) and \(e_{-1, 1} e_{1, -1} = e_{-1}\), then \(\Lambda\mathrm{sr}(\eta_1; R, \Delta) \leq k\) is equivalent to the condition \(\Lambda\mathrm{sr}(R_{11}, \Lambda e_{-1, 1}) \leq k\) from \cite{UnitStabBak1, UnitStabBak2}.

We say that \(\Lambda\mathrm{lsr}(\eta_1; R, \Delta) \leq k\) if for every maximal ideal \(\mathfrak m\) of \(K\) the inequality \(\Lambda\mathrm{sr}(\eta_1; R_{\mathfrak m}, \Delta_{\mathfrak m}) \leq k\) holds. For example, it is easy to prove that \(\Lambda\mathrm{lsr}(R, \Delta) \leq 1\) if \(R\) is quasi-finite over \(K\) (i.e. it is a direct limit of finite \(K\)-algebras).

The next proposition shows surjective stability for \(\kunit_1(R, \Delta)\) in our generality. If there are \(e_{ij} \in R_{ij}\) for all \(i, j \neq 0\) such that \(e_{ij} e_{jk} = e_{ik}\) and \(e_{ii} = e_i\), then this result already appears in \cite{OddDefPetrov}.

\begin{prop}\label{surjective-k1}
Suppose that the orthogonal hyperbolic family is free and \(\Lambda\mathrm{sr}(\eta_1; R, \Delta) \leq n - 1\). Then \(\unit(R, \Delta)\) is generated by \(\eunit(R, \Delta)\) and \(\unit_{|\eta_n|'}\).
\end{prop}
\begin{proof}
Recall that \(\unit_{|\eta_n|'}\) normalizes \(\eunit(R, \Delta) = \langle T^{\eta_n}(*), T^{-\eta_n}(*) \rangle\). Let \(g \in \unit(R, \Delta)\) be any element. First of all, suppose that \(e_n \alpha(g) e_n = e_n\) (we may consider \(e_i \alpha(g) e_i\) as \(e_i \beta(g) e_i + e_i \in R_{ii}\)). In this case there is a unique \(u \in \Delta^{|n|'}_n\) such that \(T^{\eta_n}(\dotminus u)\, g \in P_{\eta_{-n}}\) by lemma \ref{parabolic-extraction}. For such elements of the parabolic subgroup \(P_{\eta_{-n}}\) (i.e. with trivial component from \(D^{\eta_n}(*)\)) the claim is clear.

In the general case we multiply \(g\) by elementary transvections from the left until \(e_n \beta(g) e_n\) becomes \(0\). Note that the sequence \(\{e_n \inv{\alpha(g)} e_i \alpha(g) e_n\}_{i = -n}^0 \sqcup \{e_i \alpha(g) e_n\}_{i = 1}^n\) is left unimodular (where \(e_0 = 1 - \sum_{i \neq 0} e_i \in R \rtimes K\)). By \(\mathrm{sr}(R_{11}) \leq n - 1\) and properties of the stable rank there are \(b_i \in R_{in}\) for \(1 \leq i \leq n\) such that \(\{e_i \alpha(g) e_n\}_{i = -n}^{-1} \sqcup \{e_i \alpha(g) e_n + b_i \inv{\alpha(g)} e_0 \alpha(g) e_n\}_{i = 1}^n\) is left unimodular. Hence there is \(h \in \eunit(R, \Delta)\) such that \(\{e_i \alpha(hg) e_n\}_{i \neq 0}\) is left unimodular. By \(\Lambda\mathrm{sr}(R, \Delta) \leq n - 1\) there is \(h' \in \eunit(R, \Delta)\) such that \(e_n \alpha(h'hg) e_n\) is invertible in \(R_{nn}\). Then clearly there is \(h'' \in \eunit(R, \Delta)\) such that \(e_n \alpha(h''h'hg) e_n = e_n\), since \(n \geq 2\) (if \(n = 1\), then the stable rank condition implies that \(R_{nn} = 0\) and there is nothing to prove).
\end{proof}

By the main results of \cite{StabVor, Weibo}, \(\eunit(R, \Delta) \cap \unit_{|\eta_n|'}\) is generated by \(T_{ij}(a)\) and \(T_j(u)\) for \(|i|, |j| < n\) if \(\mathrm{sr}(R_{11}) \leq n - 2\). The proof actually implies surjective stability for \(\kunit_2(R, \Delta)\).

\begin{prop}\label{surjective-k2}
Suppose that the orthogonal hyperbolic family is free, \(n \geq 4\) or \(n \geq 3\) and the orthogonal hyperbolic family is strong, and \(\mathrm{sr}(R_{11}) \leq n - 2\). Then every element from \(\kunit_2(R, \Delta)\) may be generated by \(X_{ij}(a)\) and \(X_j(u)\) for \(|i|, |j| < n\).
\end{prop}
\begin{proof}
The proofs of lemmas 7 and 8 from \cite{StabVor} actually work for the Steinberg group. This means that \(\stunit(R, \Delta) = PLQ\), where \(P = \langle X_{ij}(*), X_j(*) \mid -i, j \neq -n \rangle\) is a Steinberg parabolic subgroup, \(L = \langle X_{i, -n}(*), X_{-n}(*) \mid i \geq -1 \rangle\), and \(Q = \langle X_{ij}(*), X_j(*) \mid 2 \leq j \leq n \rangle\). Now let \(plq \in \kunit_2(R, \Delta)\) for some \(p \in P\), \(l \in L\), and \(q \in Q\). It follows that \(l = 1\), hence we may assume that \(p \in \langle X_{ij}(*), X_j(*) \mid |i|, |j| < n \rangle\) (the factor from \(\langle X_{in}(*), X_n(*) \rangle\) may be pushed into \(q\)). Let \(w = X_{1n}(e_{1n})\, X_{n1}(-e_{n1})\, X_{1n}(e_{1n}) \in \stunit(R, \Delta)\) be an element swapping \(\eta_1\) with \(\eta_n\), \(q = q'd\) for \(q' \in \langle X_{ij}(*), X_j(*) \mid i \leq 1 \text{ and } 2 \leq j \rangle\) and \(d \in \langle X_{ij}(*) \mid 2 \leq i, j \leq n \rangle\). Obviously, \(q' \in \langle X_{ij}(*), X_j(*) \mid -n < i \leq 1 \text{ and } 2 \leq j < n \rangle\). It remains to prove that \(d \in \langle X_{ij}(*) \mid 1 \leq i, j \leq n \rangle\). But \(d\) commutes with \(w\) since \(d\) trivially acts on \(X_{1n}(*)\) and \(X_{n1}(*)\) (recall that \(\stmap(d) \in \unit_{|n|'}(R, \Delta)\)), and \(\up wd \in \langle X_{ij}(*) \mid 1 \leq i, j \leq n \rangle\).
\end{proof}

Now we deal with the semi-local case. Recall that a non-unital ring \(A\) is called semi-local if its factor-ring by the Jacobson radical \(\mathrm J(A)\) is unital semi-simple, where \(\mathrm J(A)\) is the set of elements generating ideals of quasi-invertible elements. For example, if \(A\) is a finite algebra over a semi-local commutative ring \(K\), then \(\mathrm J(K) A \leq \mathrm J(A)\) and \(A\) is semi-local (in this case \(A / \mathrm J(A)\) is a product of finite-dimensional semisimple algebras over the residue fields of \(K\)).
\begin{prop}\label{gauss}
Suppose that \(R\) is semi-local and \(n \geq 1\). Then there is Gauss decomposition
\[\unit(R, \Delta) = \stmap(U^+(R, \Delta; \Phi / \mathrm e_1))\, \stmap(U^-(R, \Delta; \Phi / \mathrm e_1))\, \stmap(U^+(R, \Delta; \Phi / \mathrm e_1))\, \diag(R, \Delta; \Phi / \mathrm e_1).\]
Moreover, \(\unit(R, \Delta)\) is isomorphic to the group \(G\) generated by \(T_{ij}(*)\) and \(\diag(R, \Delta; \Phi / \mathrm e_j)\) with the Steinberg relations (the ultrashort transvections are considered as elements of \(\diag(R, \Delta; \Phi / \mathrm e_j)\)) and the relations of type \(g = 1\) for every \(g\) in the free product of \(T_{ij}(*)\), \(T_{-i, j}(*)\), \(T_{i, -j}(*)\), \(T_{-i, -j}(*)\), \(\diag(R, \Delta; \Phi / \mathrm e_i)\), \(\diag(R, \Delta; \Phi / \mathrm e_j)\) for \(1 \leq i < j \leq n\) with trivial image in \(\unit(R, \Delta)\).
\end{prop}
\begin{proof}
We prove Gauss decomposition for \(n = 2\), the general case then follows by elimination of \(\mathrm e_1\) and induction. Let \(g \in \unit(R, \Delta)\) be any element. At first suppose that \(e_{-2} \alpha(g) e_{-2}\) is invertible as an element of \(R_{-2, -2}\). Then there is \(u \in \Delta^{|2|'}_{-2}\) such that \(T^{-\eta_2}(\dotminus u)\, g \in P_{\eta_2}\) by lemma \ref{parabolic-extraction}, so we are done.

In the general case we have to find \(h \in T^{\eta_2}(*) \rtimes D_2(*)\) and \(h' \in D_2(*)\) such that \(e_{-2} \alpha(hgh') e_{-2} \in R_{-2, -2}^*\). This may be done modulo the Jacobson radical of \(R\) and the kernel of \((\pi, \rho) \colon \Delta \to R \times R\), hence we may assume that \(R\) is semi-simple and \((R, \Delta)\) is special unital. Now it is possible split the second hyperbolic pair: \(\eta_2 = \bigoplus_{p = 1}^N \widetilde \eta_p\) such that all the idempotents \(e_{\pm \widetilde \eta_p}\) are primitive. Suppose that there are \(h_k \in T^{\eta_2}(*) \rtimes D_2(*)\) and \(h'_k \in D_2(*)\) such that \(\eps_{-k} \alpha(h_k g h'_k) \eps_{-k}\) is invertible in \(\eps_{-k} R \eps_{-k}\) for some \(k\), where \(\eps_{\pm k} = e_{\pm \widetilde \eta_1} + \ldots + e_{\pm \widetilde \eta_k}\) (for \(k = 0\) we may take \(h_0 = 1\)). We want to show that there is \(h_{k + 1}\) with the same property if \(k < N\), then by induction such \(h_k\) and \(h'_k\) exist for all \(k\) and we may take \(h = h_N\), \(h' = h'_N\).

Without loss of generality, \(e_{-2} \alpha(h_k g h'_k) \eps_{-k} = \eps_{-k} = \eps_{-k} \alpha(h_k g h'_k) e_{-2}\). We consider several cases, in all of them we take \(h'_{k + 1} = h'_k\). If \(v_0 = e_{-\widetilde \eta_{k + 1}} \alpha(h_k g h'_k) e_{-\widetilde \eta_{k + 1}} \neq 0\), we may take \(h_{k + 1} = h_k\). If \(v_0 = 0\) and \(v_1 = (1 - e_0 - \eps_{-k} - \eps_{k + 1}) \alpha(h_k g h'_k) e_{-\widetilde \eta_{k + 1}} \neq 0\), then we may take \(h_{k + 1} = T_{-2, \pm 1}(x)\, h_k\) or \(h_{k + 1} = D_2(x)\, h_k\) for an appropriate \(x\). If \(v_1 = 0\) and \(v_2 = (1 - e_0 - \eps_{-k} - \eps_k) \alpha(h_k g h'_k) e_{-\widetilde \eta_{k + 1}} \neq 0\), then take \(h_{k + 1} = T_{-2, \pm 1}(e_{-\widetilde \eta_{k + 1}} x)\, T_{\pm 1, 2}(y e_{\widetilde \eta_{k + 1}})\, h_k\) for appropriate \(x, y\). Assume that \(v_2 = 0\) and \(v_3 = (1 - e_0 - \eps_{-k}) \alpha(h_k g h'_k) e_{-\widetilde \eta_{k + 1}} \neq 0\). In this case we take \(h_{k + 1} = T_2(\phi(e_{-\widetilde \eta_{k + 1}} x e_{\widetilde \eta_p}))\, T_{-2, \pm 1}(e_{-\widetilde \eta_{k + 1}} y)\, T_{\pm 1, 2}(z e_{\widetilde \eta_p})\, T_{\mp 1, 2}(w e_{\widetilde \eta_{k + 1}})\, h_k\) for appropriate \(x, y, z, w\) and some \(p \leq k\) (if the division ring \(R_{\widetilde \eta_{k + 1} \widetilde \eta_{k + 1}}\) contains more than two elements, then we may take \(y = z = w = 0\)). Finally, if \(v_3 = 0\), then \(e_{-\widetilde \eta_{k + 1}} \inv{\pi(\gamma(h_k g h'_k))} e_0 \alpha(h_k g h'_k) e_{-\widetilde \eta_{k + 1}} \neq 0\) and we take \(h_{k + 1} = T_2(u \cdot e_{\widetilde \eta_{k + 1}})\, h_k\) for some \(u\).

Now we prove the second claim. Note that there is a natural map \(f \colon \stunit(R, \Delta) \to G\). We show that
\[G = f(U^+(R, \Delta; \Phi / \mathrm e_1))\, f(U^-(R, \Delta; \Phi / \mathrm e_1))\, f(U^+(R, \Delta; \Phi / \mathrm e_1))\, \diag(R, \Delta; \Phi / \mathrm e_1).\]
Denote the right hand side by \(G'\), it is a subset of \(G\). Clearly, \(G'\) is closed under multiplication by \(f(U^+(R, \Delta; \Phi / \mathrm e_1))\) and \(\diag(R, \Delta; \Phi / \mathrm e_1)\) from the right: the conjugation relations for \(T_{ij}(*)\) and \(\diag(R, \Delta; \Phi / \mathrm e_k)\) easily follow from the other relations. We show that \(G'\) is stable under the action of the Weyl group (i.e. it is preserved under permutations of \(\eta_i\) and sign changes), then obviously \(G = G'\). By definition, \(G'\) is stable under the reflection \(\mathrm e_1 \mapsto -\mathrm e_1\). By Gauss decomposition for \(\unit_{|\eta_3 \oplus \ldots \oplus \eta_n|'}\), \(G'\) is stable under the transposition \((e_1, e_2) \mapsto (e_2, e_1)\).

Recall that Gauss decomposition holds for isotropic linear groups (for example, see lemma \(6\) in \cite{LinK2}) even without Morita equivalence of the idempotents, unlike the unitary version. Using this decomposition for \((e_i + e_{i + 1}) R (e_i + e_{i + 1})\), it is easy to see that \(G'\) is stable under all transpositions \((\mathrm e_i, \mathrm e_{i + 1}) \mapsto (\mathrm e_{i + 1}, \mathrm e_i)\) for \(i > 1\). The whole Weyl group is generated by transpositions of adjacent indices and the reflection \(\mathrm e_1 \mapsto -\mathrm e_1\).

From lemma \ref{maximal-parabolic} it follows by induction that
\[\stmap(U^-(R, \Delta)) \cap \diag(R, \Delta) = \stmap(U^+(R, \Delta)) \cap \bigl(\stmap(U^-(R, \Delta)) \rtimes \diag(R, \Delta)\bigr) = 1\]
for any orthogonal hyperbolic system (even without Morita equivalence). Now if \(g \in G\) has trivial image in \(\unit(R, \Delta)\), then the factor from \(f(U^-(R, \Delta; \Phi / \mathrm e_1))\) is trivial. Hence \(g \in f(U^+(R, \Delta; \Phi / \mathrm e_1)) \diag(R, \Delta; \Phi / \mathrm e_1)\), and this group is clearly isomorphic to its image in \(\unit(R, \Delta)\).
\end{proof}

We are ready to construct an action of \(\unit(S^{-1} R, S^{-1} \Delta)\) on \(\stunit^{(\infty, S)}(R, \Delta)\) in the interesting cases.

\begin{theorem}\label{local-action}
Let \(K\) be a commutative ring, \((R, \Delta)\) be an odd form \(K\)-algebra with an orthogonal hyperbolic family of rank \(n\), \(S \subseteq K\) be a multiplicative subset. Suppose that \(n \geq 4\) or \(n \geq 3\) and the orthogonal hyperbolic family is strong. Suppose also that \(S^{-1} R\) is semi-local or the orthogonal hyperbolic family is free and \(\Lambda\mathrm{sr}(\eta_1; S^{-1} R, S^{-1} \Delta) \leq n - 2\). Then \(\unit(S^{-1} R, S^{-1} \Delta)\) acts on \(\stunit^{(\infty, S)}(R, \Delta)\) making \(\stmap\) equivariant, this action is the usual one on every \(\diag(S^{-1} R, S^{-1} \Delta; \Phi / \alpha)\).
\end{theorem}
\begin{proof}
In the semi-local case this follows from proposition \ref{gauss} and lemma \ref{root-generation}. In the stable rank case note that \(\unit(S^{-1} R, S^{-1} \Delta)\) is generated by \(\unit_{|S^{-1} \eta_{n - 1} \oplus S^{-1} \eta_n|'} \leq \unit(S^{-1} R, S^{-1} \Delta)\) and \(\stunit(S^{-1} R, S^{-1} \Delta)\) by proposition \ref{surjective-k1}. The semi-direct product of these groups acts on \(\stunit^{(\infty, S)}(R, \Delta)\) by lemma \ref{root-generation} and propositions \ref{short-presentation}, \ref{ultrashort-presentation}. We have to show that every element \((g, h) \in \stunit(S^{-1} R, S^{-1} \Delta) \rtimes \unit_{|S^{-1} \eta_{n - 1} \oplus S^{-1} \eta_n|'}\) with trivial image in \(\unit(R, \Delta)\) acts trivially. By injective stability of \(\kunit_1\) and surjective stability of \(\kunit_2\) (proposition \ref{surjective-k2}) we may assume that \(g\) lies in the image of \(\stunit(S^{-1} R_{|\eta_n|', |\eta_n|'}, S^{-1} \Delta_{|\eta_n|'}^{|\eta_n|'})\). But then \((g, h)\) acts trivially on \(\stunit^{(\infty, S)}(R, \Delta; \Phi / \{\mathrm e_1, \ldots, \mathrm e_{n - 1}\})\), hence on \(\stunit^{(\infty, S)}(R, \Delta; \Phi)\) by lemma \ref{root-generation}.
\end{proof}

\section{Unitary Steinberg crossed module}

Now we prove the main results. Recall that the Steinberg group is perfect for \(n \geq 3\) by lemma \ref{perfect}.

\begin{theorem}\label{semilocal-crossed-module}
Let \((R, \Delta)\) be an odd form ring with an orthogonal hyperbolic family of rank \(n \geq 3\). Suppose that \(R\) is semi-local. Then there is a unique action of \(\unit(R, \Delta)\) on \(\stunit(R, \Delta)\) making \(\stmap\) a crossed module, it is consistent with the action of \(\stunit(R, \Delta) \rtimes \diag(R, \Delta)\).
\end{theorem}
\begin{proof}
Clearly, \(\stunit(R, \Delta)\) and \(\diag(R, \Delta; \Phi / \alpha)\) for all ultrashort \(\alpha\) act on \(\stunit(R, \Delta)\) making \(\stmap\) equivariant by proposition \ref{ultrashort-presentation}. By lemma \ref{root-generation} and proposition \ref{gauss}, these actions glue together to the required action of \(\unit(R, \Delta)\). Now \(\stunit(R, \Delta) \to \eunit(R, \Delta)\) is a central perfect extension, so the uniqueness follows from abstract group theory.
\end{proof}

The following lemma simplifies the arguments in the proof of our main result. For a maximal ideal \(\mathfrak m\) of \(K\) we use the index \((\infty, \mathfrak m)\) instead of \((\infty, K \setminus \mathfrak m)\).

\begin{lemma}\label{costalks}
Let \(K\) be a commutative ring, \((M, M_0)\) be a \(2\)-step nilpotent \(K\)-module, \(S \subseteq K\) be a multiplicative subset, and \(G^{(\infty)}\) be a pro-group. Then any morphism of pro-groups \(M^{(\infty, S)} \to G^{(\infty)}\) is uniquely determined by its compositions with \(M^{(\infty, \mathfrak m)} \to R^{(\infty, S)}\) for all maximal ideals \(\mathfrak m\) disjoint with \(S\).
\end{lemma}
\begin{proof}
Without loss of generality, \(G\) is an ordinary group. Let \(f_1, f_2 \colon M^{(s)} \to G\) be two homomorphisms for some \(s \in S\) such that \(f_1\bigl(m^{(\infty)}\bigr) = f_2\bigl(m^{(\infty)}\bigr)\) for \(m^{(\infty)} \in M^{(\infty, \mathfrak m)}\), where \(\mathfrak m\) runs over maximal ideals of \(K\) disjoint with \(S\). Then the set
\[\mathfrak a = \{k \in K \mid f_1\bigl(\iota(k M_0^{(s)})\bigr) = f_2\bigl(\iota(k M_0^{(s)})\bigr)\}\]
is an ideal of \(K\). It is not contained in any maximal ideal \(\mathfrak m\) disjoint with \(S\) by assumption, hence \(S^{-1} \mathfrak a = S^{-1} K\). It follows that \(\mathfrak a\) intersects \(S\), i.e. \(f_1\) and \(f_2\) are equal after a restriction to \(M_0^{(ss')}\) for some \(s' \in S\).

Without loss of generality, \(f_1\) and \(f_2\) coincide on \(M_0^{(s)}\). Then the set
\[\mathfrak b = \{k \in K \mid f_1\bigl(M^{(s)} \cdot k\bigr) = f_2\bigl(M^{(s)} \cdot k\bigr)\}\]
is also an ideal of \(K\). By the reason as above, it intersects \(S\), so \(f_1\) and \(f_2\) are equal after a restriction to some \(M^{(ss')}\) for some \(s' \in S\).
\end{proof}

\begin{theorem}\label{steinberg-crossed-module}
Let \(K\) be a commutative ring, \((R, \Delta)\) be an odd form \(K\)-algebra with an orthogonal hyperbolic family of rank \(n\). Suppose that \(n \geq 4\) or \(n \geq 3\) and the orthogonal hyperbolic family is strong. Suppose also that \(R\) is quasi-finite over \(K\) or \(\Lambda\mathrm{lsr}(\eta_1, R, \Delta) \leq n - 2\). Then there is a unique action of \(\unit(R, \Delta)\) on \(\stunit(R, \Delta)\) making \(\stmap\) a crossed module, it is consistent with the action of \(\stunit(R, \Delta) \rtimes \diag(R, \Delta)\) and with the actions.
\end{theorem}
\begin{proof}
Note that the quasi-finite case easily reduces to the case when \(R\) is finite \(K\)-algebra. First of all, we show that any \(g \in \kunit_2(R, \Delta)\) lies in the center of \(\stunit(R, \Delta)\). Note that the conjugation by \(g\) and the identity are the same automorphism on \(\stunit^{(\infty, \mathfrak m)}(R, \Delta)\) for every maximal ideal \(\mathfrak m\) of \(R\) by theorem \ref{local-action}. Hence \(g\) stabilizes each root subgroup by lemma \ref{costalks}. In other words, \(\stunit(R, \Delta)\) is a perfect central extension of \(\eunit(R, \Delta)\).

Next, we show that \(\eunit(R, \Delta)\) is normalized by any \(g \in \unit(R, \Delta)\) (in the matrix case this is the main result of \cite{OddDefPetrov}). We use a variant of the argument used in the proof of lemma \ref{costalks}. For any non-zero indices \(i \neq \pm j\) let
\begin{align*}
\mathfrak a &= \{k \in K \mid \up g{T_{ij}(kR_{ij})} \leq \eunit(R, \Delta)\},
\end{align*}
it is an ideal. By theorem \ref{local-action}, it is not contained in any maximal ideal of \(K\), hence \(\mathfrak a = K\). By the same argument (but with two ideals), \(\up g{T_j(*)} \leq \eunit(R, \Delta)\).

It remains to prove that for each \(g \in \unit(R, \Delta)\) there is an endomorphism \(\stunit(R, \Delta) \to \stunit(R, \Delta), X_{ij}(a) \mapsto \up g{X_{ij}(a)}, X_j(u) \mapsto \up g{X_j(u)}\) making \(\stmap\) equivariant. Note that since \(\stunit(R, \Delta) \to \eunit(R, \Delta)\) is a central perfect extension, such endomorphisms are unique for all \(g\) (if they exist) and they are multiplicative on \(g\). We are going to show that such an endomorphism exists and, moreover, the morphisms \(\stunit^{(\infty, \mathfrak m)}(R, \Delta) \to \stunit(R, \Delta)\) of pro-groups are equivariant under the action of \(g\) for every maximal ideal \(\mathfrak m\) of \(K\).

Let \(\mathcal Y_{ij}(a) = \stmap^{-1}(\up g{T_{ij}(a)})\) and \(\mathcal Y_j(u) = \stmap^{-1}(\up g{T_j(u)})\), they are certain cosets of \(\kunit_2(R, \Delta)\). Clearly, \(\mathcal Y_{ij}(a + b) = \mathcal Y_{ij}(a)\, \mathcal Y_{ij}(b)\) and \(\mathcal Y_j(u \dotplus v) = \mathcal Y_j(u)\, \mathcal Y_j(v)\). Note that all commutators of these cosets are one-element sets, so we may consider them as elements of the Steinberg group. Also for every maximal ideal \(\mathfrak m\) of \(K\) we have
\begin{align*}
\pi_{\stmap}\bigl(\up g{X_{ij}\bigl(a^{(\infty)}\bigr)}\bigr) &\in \mathcal Y_{ij}\bigl(\pi_{\mathrm{alg}}\bigl(a^{(\infty)}\bigr)\bigr),\\
\pi_{\stmap}\bigl(\up g{X_j\bigl(u^{(\infty)}\bigr)}\bigr) &\in \mathcal Y_j\bigl(\pi_{\mathrm{alg}}\bigl(u^{(\infty)}\bigr)\bigr)
\end{align*}
in the obvious sense (the right hand sides are morphisms from pro-groups to the set of subsetes of \(\stunit(R, \Delta)\)), where \(\pi_{\stmap} \colon \stunit^{(\infty, \mathfrak m)}(R, \Delta) \to \stunit(R, \Delta)\) and \(\pi_{\mathrm{alg}} \colon (R^{(\infty, \mathfrak m)}, \Delta^{(\infty, \mathfrak m)}) \to (R, \Delta)\) are the canonical morphisms.

Consider non-zero indices \(i, j, k, l\) such that \(i \neq l \neq -j \neq -k \neq i\). For any \(a \in R_{ij}\) consider
\[\mathfrak b = \{k \in K \mid [\mathcal Y_{ij}(a), \mathcal Y_{kl}(kR_{kl})] = 1\},\]
it is an ideal of \(K\). Since
\[\bigl[\mathcal Y_{ij}(a), \mathcal Y_{kl}\bigl(\pi_{\mathrm{alg}}\bigl(b^{(\infty)}\bigr)\bigr)\bigr] = \pi_{\stmap}\bigl(\up{g X_{ij}(a)}{X_{kl}\bigl(b^{(\infty)}\bigr)}\, \up g{X_{kl}\bigl(-b^{(\infty)}\bigr)}\bigr) = 1\]
for \(b^{(\infty)} \in R_{kl}^{(\infty, \mathfrak m)}\) and a maximal ideal \(\mathfrak m\) of \(K\), by theorem \ref{local-action} this ideal equals \(K\). In other words, the maps \(\mathcal Y_\alpha\) satisfy (St3). Similarly, they satisfy (St7).

Let \(i \neq \pm j\) be non-zero indices and choose an index \(k \neq 0, \pm i, \pm j\). The maps
\[Y_{ij}^{\pm k} \colon R_{i, \pm k} \times R_{\pm k, j} \to \stunit(R, \Delta), (a, b) \mapsto [\mathcal Y_{i, \pm k}(a), \mathcal Y_{\pm k, j}(b)] \in \mathcal Y_{ij}(ab)\]
satisfy all the axioms from lemma \ref{ring-presentation}: the first three are now obvious and the last one follows from lemma \ref{costalks} if we fix any two parameters and let the third vary. Hence there is a well-defined homomorphism \(Y_{ij} \colon R_{ij} \to \stunit(R, \Delta)\) restricting to \(Y_{ij}^{\pm k}\) and such that
\[Y_{ij}\bigl(\pi_{\mathrm{alg}}\bigl(a^{(\infty)}\bigr)\bigr) = \pi_{\stmap}\bigl(\up g{X_{ij}\bigl(a^{(\infty)}\bigr)}\bigr)\]
for \(a^{(\infty)} \in R_{ij}^{(\infty, \mathfrak m)}\) and any maximal ideal \(\mathfrak m\) of \(K\) (by lemma \ref{ring-generation}). It is easy to see using lemma \ref{costalks} that they satisfy (St0) and (St4).

Now let \(i \neq 0\) and choose \(j \neq 0, \pm i\). Similarly to the short root case, there exists a homomorphism \(Y_{-i, i} \colon R_{-i, i} \to \stunit(R, \Delta)\) such that
\[Y_{-i, i}\bigl(\pi_{\mathrm{alg}}\bigl(a^{(\infty)}\bigr)\bigr) = \pi_{\stmap}\bigl(\up g{X_i\bigl(\phi(a^{(\infty)})\bigr)}\bigr)\]
for \(a^{(\infty)} \in R_{-i, i}^{(\infty, \mathfrak m)}\) and any maximal ideal \(\mathfrak m\) of \(K\), it satisfies (St5) and the identity \(Y_{-i, i}(a) = Y_{-i, i}(-\inv a)\). The maps
\[Y_i^{\pm j} \colon \Delta^0_{\pm j} \times R_{\pm j, i} \to \stunit(R, \Delta), (u, b) \mapsto Y_{\mp j, i}(\inv{\rho(u)} b)\, [\mathcal Y_{\pm j}(\dotminus u), Y_{\pm j, i}(-b)] \in \mathcal Y_i(u \cdot b)\]
together with \(Y_{-i, i}\) satisfy the axioms from lemma \ref{form-presentation} (the non-obvious axioms follow by lemma \ref{costalks} applied to any parameter). Hence there is a homomorphism \(Y_i \colon \Delta^0_i \to \stunit(R, \Delta)\) restricting to \(Y_i^{\pm j}\) and \(Y_{-i, i}\), and in addition
\[Y_i\bigl(\pi_{\mathrm{alg}}\bigl(u^{(\infty)}\bigr)\bigr) = \pi_{\stmap}\bigl(\up g{X_i\bigl(u^{(\infty)}\bigr)}\bigr)\]
for \(u^{(\infty)} \in \Delta^{0, (\infty, \mathfrak m)}_i\) and any maximal ideal \(\mathfrak m\) of \(K\) (by lemma \ref{form-generation}). By lemma \ref{costalks}, it satisfies the remaining Steinberg relations.
\end{proof}

\bibliographystyle{plain}  
\bibliography{references}

\end{document}